\documentclass[11pt]{article}

\usepackage{amsmath, amssymb, amsthm} 
\usepackage{centernot}
\usepackage{mathpazo}
\usepackage[margin=2.5cm]{geometry}
\usepackage{subcaption}
\usepackage{color, graphicx, soul}
\usepackage{adjustbox}

\usepackage{resizegather}

\usepackage{xargs}

\usepackage{hyperref}

\usepackage{empheq}
\definecolor{myblue}{rgb}{.9, .9, 1}
\newcommand*\mybluebox[1]{%
	\colorbox{myblue}{\hspace{1em}#1\hspace{1em}}
}

\usepackage[capitalize]{cleveref}
\crefname{equation}{}{equations}
\crefname{chapter}{Appendix}{chapters}
\crefname{item}{}{items}
\crefname{figure}{Figure}{figures}
\makeatletter
\def\namedlabel#1#2{\begingroup
   \def\@currentlabel{#2}%
   \label{#1}\endgroup
}
\makeatother

\makeatletter
\def\th@plain{%
  \thm@notefont{}
  \itshape 
}
\def\th@definition{%
  \thm@notefont{}
  \normalfont 
}
\makeatother

\newtheorem{theorem}{Theorem}[section]
\newtheorem{lemma}[theorem]{Lemma}
\newtheorem{corollary}[theorem]{Corollary}
\newtheorem{proposition}[theorem]{Proposition}
\newtheorem{fact}[theorem]{Fact}

\theoremstyle{definition}
\newtheorem{definition}[theorem]{Definition}
\theoremstyle{definition}
\newtheorem{example}[theorem]{Example}
\theoremstyle{definition}
\newtheorem{remark}[theorem]{Remark}

\usepackage[shortlabels]{enumitem}
\setlist[enumerate]{nosep}

\allowdisplaybreaks

\newcommand{\sepp}{\setlength{\itemsep}{-2pt}}

\newcommand{\menge}[2]{\{{#1}~\big |~{#2}\}}

\newcommand{\scal}[2]{\left\langle {#1},{#2} \right\rangle}

\newcommand{\nnn}{\ensuremath{{n\in{\mathbb N}}}}

\newcommand{\RR}{\ensuremath{\mathbb R}}
\newcommand{\RP}{\ensuremath{\mathbb{R}_+}}
\newcommand{\RPP}{\ensuremath{\mathbb{R}_{++}}}

\newcommand{\argmin}{\ensuremath{\operatorname*{argmin}}}

\newcommand{\inte}{\ensuremath{\operatorname{int}}}

\newcommand{\ran}{\ensuremath{\operatorname{ran}}}

\newcommand{\dom}{\ensuremath{\operatorname{dom}}}
\newcommand{\intdom}{\ensuremath{\operatorname{int}\operatorname{dom}}\,}

\newcommand{\Fix}{\ensuremath{\operatorname{Fix}}}
\newcommand{\Id}{\ensuremath{\operatorname{Id}}}
\newcommand{\prox}{\ensuremath{\operatorname{Prox}}}

\newcommand{\infconv}{\ensuremath{\mbox{\small$\,\square\,$}}}
\newcommand{\lev}[1]{\ensuremath{\mathrm{lev}_{\leq #1}\:}}
\newcommand{\fenv}[2][]%
{\ensuremath{\,\overrightarrow{\operatorname{env}}_{#2}}^{#1}}
\newcommand{\benv}[2][]%
{\ensuremath{\,\overleftarrow{\operatorname{env}}_{#2}^{#1}}}
\newcommand{\env}[2][]%
{\ensuremath{{\operatorname{env}}_{#2}^{#1}}}

\newcommand{\bP}[2][]%
{\ensuremath{\,\overleftarrow{\operatorname{P}}_{#2}^{#1}}}

\newcommand{\fP}[2][]%
{\ensuremath{\,\overrightarrow{\operatorname{P}}_{#2}^{#1}}}

\newcommand{\fprox}[1]{\overrightarrow{\operatorname{P}}_%
{\negthinspace\negthinspace#1}}
\newcommand{\bprox}[1]{\overleftarrow{\operatorname{P}}_%
{\negthinspace\negthinspace #1}}
\newcommand{\fproj}[1]{\overrightarrow{\operatorname{P}\thinspace}_%
{\negthinspace\negthinspace #1}}
\newcommand{\proj}[1]{{\operatorname{P}\thinspace}_%
{\negthinspace\negthinspace #1}}
\newcommand{\bproj}[1]{\overleftarrow{\thinspace \operatorname{P}\thinspace}_%
{\negthinspace\negthinspace #1}}

\begin{document}

\title{Regularizing with Bregman--Moreau envelopes}

\author{
Heinz H.\ Bauschke\thanks{
Mathematics, University of British Columbia, Kelowna, B.C.\ V1V~1V7, Canada. 
E-mail: \texttt{heinz.bauschke@ubc.ca}.},~
Minh N.\ Dao\thanks{
CARMA, University of Newcastle, Callaghan, NSW 2308, Australia.
E-mail: \texttt{daonminh@gmail.com}.}~~~and~
Scott B.\ Lindstrom\thanks{
CARMA, University of Newcastle, Callaghan, NSW 2308, Australia.
E-mail: \texttt{scott.lindstrom@uon.edu.au}.}
}

\date{November 16, 2018}

\maketitle

\begin{abstract} \noindent
Moreau's seminal paper, 
introducing what is now called the Moreau envelope and the
proximity operator (also known as the proximal mapping), appeared in 1962. 
The Moreau envelope of a given convex function provides a regularized
version which has additional desirable properties such as differentiability
and full domain. Forty years ago, Attouch proposed using
the Moreau envelope for regularization. Since then, this branch
of convex analysis has developed in many fruitful directions.
In 1967,
Bregman introduced what is nowadays known as the Bregman distance
as a measure of discrepancy between two points generalizing
the square of the Euclidean distance.
Proximity operators based on the Bregman distance have become a
topic of significant research as they are useful in the algorithmic
solution of optimization problems. 
More recently, 
in 2012, Kan and Song studied regularization aspects of the
left Bregman--Moreau envelope even for nonconvex functions. 
In this paper, we complement previous works by analyzing 
the left and right Bregman--Moreau envelopes 
and by providing additional asymptotic results. 
Several examples are provided. 
\end{abstract}

{\small 

}

\section{Introduction}

We assume throughout that
\begin{empheq}[box=\mybluebox]{equation}
\text{$X := \RR^J$,}
\end{empheq}
which we equip with 
the standard inner product $\scal{\cdot}{\cdot}$ and 
the induced Euclidean norm $\|\cdot\|$.

Let $\theta\colon X\to \left]-\infty, +\infty\right]$ be convex, lower semicontinuous, and
proper\footnote{See \cite{Roc70}, \cite{BC17}, 
\cite{MorNam}, and 
\cite{RW98}
for
background material in convex analysis from which we adopt our
notation which is standard. 
We also set  $\RPP := \menge{x \in \RR}{x >0}$.}.
The \emph{Moreau envelope} with parameter $\gamma\in\RPP$ is the
function
\begin{equation}
\label{e:170510a}
\env[\gamma]{\theta}\colon x\mapsto \inf_{y\in X} \big(\theta(y) +
\frac{1}{2\gamma}\|x-y\|^2\big). 
\end{equation}
Moreau only considered the case in which $\gamma=1$;
the systematic study involving the parameter $\gamma$ originated
with 
Attouch (see \cite{Attouch77} and \cite{Attouch84}).
If $\theta=\iota_C$, the indicator function of a
nonempty closed convex subset $C$ of $X$,
then the corresponding Moreau envelope with parameter $\gamma$ is
$\tfrac{1}{2\gamma}d_C^2$, where $d_C$ is the distance function
of the set $C$. 
While the indicator function has (effective) domain $C$ and 
is differentiable
only on $\inte C$, the interior of $C$, the Moreau envelope 
is much better behaved: 
for instance, it has full domain and is 
differentiable everywhere.

Now assume that 
\begin{empheq}[box=\mybluebox]{equation}
\text{$f \colon X \to\left]-\infty, +\infty\right]$ is convex and
differentiable on $U := \intdom f \neq\varnothing$.}
\end{empheq}
The \emph{Bregman distance}\footnote{Note that $D_f$ is
\emph{not} a distance in the sense of metric topology; however,
this naming convention is now ubiquitous.}  associated with $f$, 
first explored by Bregman in \cite{Bre67} (see also
\cite{CZ97}), is 
\begin{equation} 
\label{e:Df}
D_f \colon X \times X \to \left[0, +\infty\right] \colon (x,y) \mapsto
\begin{cases}
f(x) -f(y) -\scal{\nabla f(y)}{x -y}, &\text{~if~} y\in U; \\
+\infty, &\text{~otherwise}.
\end{cases}
\end{equation} 
It serves as a measure of discrepancy between two points and thus
gives rise to associated projectors (nearest-point mappings)
and proximal mappings
which have been employed to solve convex feasibility and
optimization 
problems algorithmically; see, e.g.,
\cite{ACM}, 
\cite{BBT}, 
\cite{BBC03},
\cite{BC03},
\cite{BCN06},
\cite{BauLew},
\cite{BN02},
\cite{BurKas12}, 
\cite{ButIus},
\cite{BC01},
\cite{CH02},
\cite{CR98},
\cite{CZ92},
\cite{CZ97},
\cite{CT93},
\cite{CKS},
\cite{CN16},
\cite{Eck93},
\cite{KS},
\cite{KRS}, 
\cite{Kiw97},
\cite{Nguyen17}, 
\cite{Nguyen16}, 
and
\cite{Sabach11}. 
The classical case arises when $f=\tfrac{1}{2}\|\cdot\|^2$ in
which case $D_f(x,y)= \tfrac{1}{2}\|x-y\|^2=D_f(y,x)$. 
This clearly suggests replacing the quadratic term in
\eqref{e:170510a} by the Bregman distance. 
However, because different assignments of $f$ may allow for
cases in which 
$D_f(x,y)\neq D_f(y,x)$, we actually are led to consider
\emph{two} envelopes:
the \emph{left} and \emph{right} 
\emph{Bregman--Moreau envelopes} are defined by
\begin{equation}
\label{e:bba}
\benv[\gamma]{\theta} \colon X\to \left[-\infty, +\infty\right] \colon
y\mapsto \inf_{x\in X} \big(\theta(x) +\frac{1}{\gamma}D_f(x,y)\big)
\end{equation}
and 
\begin{equation}
\label{e:fba}
\fenv[\gamma]{\theta} \colon X \to \left[-\infty, +\infty\right] \colon
x \mapsto
\inf_{y\in X} \big(\theta(y) +\frac{1}{\gamma}D_f(x,y)\big),
\end{equation}
respectively. 
It follows from the definition (see also \cref{ex:examples} below) 
that if $f=\frac12\|\cdot\|^2$, then $D_f\colon(x,y)\mapsto\frac12\|x-y\|^2$, and 
$\benv[\gamma]{\theta}=\fenv[\gamma]{\theta}=\theta\infconv(\frac{1}{2\gamma}\|\cdot\|^2)$ 
is the classical \emph{Moreau envelope} of $\theta$ of parameter $\gamma$; 
see \cite{Mor62}, \cite{Mor65}, and also \cite[Section~12.4]{BC17} and \cite[Section~1.G]{RW98}. 
When $\gamma =1$, 
we simply write $\benv{\theta}$ for $\benv[1]{\theta}$, 
and $\fenv{\theta}$ for $\fenv[1]{\theta}$, which were introduced in \cite{BCN06}. 
Bregman--Moreau envelopes when $\gamma\neq 1$ were
previously explored 
in \cite{CKS} and \cite{KS} for the \emph{left} variant; the authors
provided asymptotic results when $\gamma\downarrow 0$.

{\it The goal of this paper is to present a systematic study of
regularization aspects of the Bregman--Moreau envelope.}
Our results extend and complement several classical results 
and provide a novel
way to approximate $\theta$. 
We also obtain new results on the asymptotic behaviour 
when $\gamma\uparrow+\infty$
and on the \emph{right} Bregman--Moreau envelope. 
This opens the door to regularization and smoothing
of functions by employing the right Bregman--Moreau envelope. We
also provide visualizations and examples.

The remainder of this paper is organized as follows.
In Section~\ref{s:zwei}, we collect various useful properties and
characterizations of Bregman--Moreau envelopes. 
In particular,
the minimizers of the envelopes are
also minimizers of the original function
(see \cref{p:min}). 
Section~\ref{s:drei} is devoted to the asymptotic behaviour of
the Bregman--Moreau envelopes when 
$\gamma\downarrow 0$ (\cref{p:12.32}) and when
$\gamma\uparrow +\infty$ (\cref{p:limprox+}). 
Finally, Section~\ref{s:vier} provides examples and comments on
future work.

\section{Basic properties}

\label{s:zwei}

In this section, we collect various useful properties of the
Bregman--Moreau envelopes. 

We start by describing the effect of scaling the function. 

\begin{proposition}
\label{p:12.22i}
Let $\theta \colon X \to \left]-\infty, +\infty\right]$, 
let $\gamma \in \RPP$, and let $\mu \in \RPP$.
Then $\benv[\mu]{\gamma\theta} =\gamma\benv[\gamma\mu]{\theta}$ 
and $\fenv[\mu]{\gamma\theta} =\gamma\fenv[\gamma\mu]{\theta}$.
\end{proposition}
\begin{proof}
This is analogous to the proof of \cite[Proposition~12.22(i)]{BC17}.
\end{proof}

We now turn to regularization properties.
(For a variant of \cref{p:12.9}\ref{p:12.9_left},
see \cite[Theorem~2.2 and Proposition~2.1(i)]{KS}.)

\begin{proposition} 
\label{p:12.9} 
Let $\theta \colon X \to \left]-\infty, +\infty\right]$ be such that $U\cap \dom\theta\neq\varnothing$ and let $\gamma \in \RPP$.
Then the following hold:
\begin{enumerate}
\item 
\label{p:12.9_left}
$\dom\benv[\gamma]{\theta}=U$, 
and $(\forall y\in U)(\forall\mu \in \left]\gamma, +\infty\right[)$ 
$\inf \theta(X) \leq \benv[\mu]{\theta}(y) \leq \benv[\gamma]{\theta}(y) \leq \theta(y)$.
Consequently, $\inf \theta(X) \leq \inf \benv[\gamma]{\theta}(X) \leq \inf \theta(U)$, 
and $\benv[\gamma]{\theta}(y) \downarrow \inf \theta(X)$ as $\gamma \uparrow +\infty$.
\item 
\label{p:12.9_right}
$\dom\fenv[\gamma]{\theta}=\dom f$, 
and $(\forall x \in U)(\forall\mu \in \left]\gamma, +\infty\right[)$ 
$\inf \theta(X) \leq \fenv[\mu]{\theta}(x) \leq \fenv[\gamma]{\theta}(x) \leq \theta(x)$.
Consequently, $\inf \theta(X) \leq \inf \fenv[\gamma]{\theta}(X) \leq \inf \theta(U)$, 
and $\fenv[\gamma]{\theta}(x) \downarrow \inf \theta(X)$ as $\gamma \uparrow +\infty$.
\end{enumerate}
\end{proposition}
\begin{proof}
\ref{p:12.9_left}: We first show that $\dom\benv[\gamma]{\theta}=U$.
Let $y \in \dom\benv[\gamma]{\theta}$. 
Then $\benv[\gamma]{\theta}(y) =\inf_{x\in X} \big(\theta(x) +\frac{1}{\gamma}D_f(x,y)\big) <+\infty$,
and hence there exists $x \in X$ such that $\theta(x) +\frac{1}{\gamma}D_f(x,y) <+\infty$.
Since $\theta(x) >-\infty$, this yields $y \in U$.

From now on, let $y \in U$, and pick $u \in \dom\theta \cap U$. 
Then $-f(y) <+\infty$, $\|\nabla f(y)\| <+\infty$, $f(u) <+\infty$, $\theta(u) <+\infty$, and
\begin{subequations}
\begin{align}
\benv[\gamma]{\theta}(y) &=\inf_{x\in X} \Big(\theta(x) +\frac{1}{\gamma}\big(f(x) -f(y) -\scal{\nabla f(y)}{x -y}\big)\Big), \\ 
&\leq \theta(u) +\frac{1}{\gamma}\big(f(u) -f(y) -\scal{\nabla f(y)}{u -y}\big) <+\infty,
\end{align}
\end{subequations} 
which gives $y \in \dom\benv[\gamma]{\theta}$. Hence, $\dom \benv[\gamma]{\theta} =U$. 

Next, let $\mu \in \left]\gamma, +\infty\right[$. 
Then $\frac{1}{\mu} < \frac{1}{\gamma}$, 
$\theta \leq \theta +\frac{1}{\mu}D_f(\cdot,y) \leq \theta +\frac{1}{\gamma}D_f(\cdot,y)$, and so
\begin{equation}
\inf_{x \in X} \theta(x) \leq \benv[\mu]{\theta}(y) \leq \benv[\gamma]{\theta}(y) =\inf_{x \in X} \big(\theta(x) +\frac{1}{\gamma}D_f(x,y)\big)
\leq \theta(y) +\frac{1}{\gamma}D_f(y,y) =\theta(y). 
\end{equation}
Therefore,
\begin{equation}
\label{e:170428a}
\inf \theta(X) \leq \benv[\mu]{\theta}(y) \leq \benv[\gamma]{\theta}(y) \leq \theta(y).
\end{equation}
Taking now the infimum over $y\in U$ yields
$\inf \theta(X) \leq \inf \benv[\gamma]{\theta}(X) \leq \inf \theta(U)$.
Consequently, 
\begin{equation}
\inf \theta(X) \leq \varliminf_{\gamma \to+\infty}
\benv[\gamma]{\theta}(y).
\end{equation}
On the other hand, $(\forall x \in X)$ $\benv[\gamma]{\theta}(y) \leq \theta(x) +\frac{1}{\gamma}D_f(x,y)$, 
which implies that $(\forall x \in X)$ $\varlimsup_{\gamma
\to+\infty} \benv[\gamma]{\theta}(y) \leq \theta(x)$ and thus
$\varlimsup_{\gamma \to+\infty} \benv[\gamma]{\theta}(y) \leq
\inf \theta(X)$. 
Altogether,
$\lim_{\gamma\to+\infty} \benv[\gamma]{\theta}(y) =
\inf \theta(X)$ and the conclusion follows from
\eqref{e:170428a}. 

\ref{p:12.9_right}: This is similar to \ref{p:12.9_left}. 
\end{proof}

Denote by $\Gamma_0(X)$ the set of all proper lower semicontinuous convex functions from $X$ to $\left]-\infty, +\infty\right]$. 
From now on, we strengthen our assumptions by requiring that 
\begin{empheq}[box=\mybluebox]{equation}
\label{e:fLegendre}
\text{$f\in\Gamma_0(X)$ is a convex function of Legendre type and
$U := \intdom f$.}
\end{empheq}
This will allow us to obtain a quite satisfying theory in which
the envelopes are convex functions. 
Note that $f$ is essentially smooth and essentially strictly convex in the sense of \cite[Section~26]{Roc70}.
It is well known that
\begin{equation}
\label{e:Leggy}
\nabla f\colon U\to U^*:= \intdom f^*
\text{~is a homeomorphism with~}
\big(\nabla f\big)^{-1} = \nabla f^*.
\end{equation}
We will also work with the following standard assumptions:
\begin{itemize}
\item[\bfseries A1] 
$\nabla^2 f$ exists and is continuous on $U$;
\item[\bfseries A2] 
$D_f$ is \emph{jointly convex}, i.e., convex on $X\times X$;
\item[\bfseries A3] 
$(\forall x\in U)$ $D_f(x,\cdot)$ is strictly convex on $U$;
\item[\bfseries A4] 
$(\forall x\in U)$ $D_f(x,\cdot)$ is coercive, i.e., 
$D_f(x, y) \to +\infty$ as $\|y\| \to +\infty$.
\end{itemize}
We also henceforth assume that 
\begin{empheq}[box=\mybluebox]{equation}
\text{all assumptions \textbf{A1}--\textbf{A4} hold.}
\end{empheq}

\begin{example}[see {\cite[Example~2.16]{BN02}}] 
\label{ex:examples}
Assumptions \eqref{e:fLegendre} and 
\textbf{A1}--\textbf{A4} hold in the following cases, 
where $x=(\xi_j)_{1\leq j\leq J}$ and $y=(\eta_j)_{1\leq j\leq J}$ are 
two generic points in $X =\RR^J$.
\begin{enumerate}
\item 
\emph{Energy:}
If $f\colon x\mapsto\tfrac{1}{2}\|x\|^2$, then $U=X$ and 
\begin{equation}
D_f(x,y) = \tfrac{1}{2}\|x-y\|^2.
\end{equation}
\item  
\label{ex:examples:KL}
\emph{Boltzmann--Shannon\footnote{When dealing with the Boltzmann--Shannon entropy and Fermi--Dirac 
entropy, it is understood that 
$0 \cdot \ln(0) := 0$. 
For two vectors $x$ and $y$ in $X$,
expressions
such as $x \leq y$, $x \cdot y$, and $x/y$
are interpreted coordinate-wise.} entropy:}
If $f\colon x\mapsto\displaystyle\sum_{j=1}^{J}\xi_j\ln(\xi_j)-\xi_j$, 
then $U =\menge{x\in X}{x >0}$
and one obtains the \emph{Kullback--Leibler divergence}
\begin{equation}
D_f(x,y) = \begin{cases}
\textstyle \sum_{j=1}^J \xi_j \ln(\xi_j/\eta_j) - \xi_j + \eta_j, & 
\text{if $x \geq 0$ and $y>0$;}\\
+\infty, & \text{otherwise.}
\end{cases}
\end{equation}
\item 
\emph{Fermi--Dirac entropy:} 
If $f\colon x\mapsto \displaystyle\sum_{j=1}^{J}\xi_j\ln(\xi_j)+(1-\xi_j)\ln(1-\xi_j)$, then $U =\menge{x \in X}{0 <x <1}$ and 
\begin{equation}
\resizebox{.9\hsize}{!}{%
$D_f(x,y) = \begin{cases}
\textstyle\sum_{j=1}^J
\xi_j\ln(\xi_j/\eta_j)+(1-\xi_j)\ln\big((1-\xi_j)/(1-\eta_j)\big),&
\text{if $0 \leq x \leq 1$ and $0 <y <1$;}\\
+\infty, & \text{otherwise.}
\end{cases}$%
}
\end{equation}
\end{enumerate}
\end{example}

The following result relates the Bregman--Moreau envelopes to
Fenchel conjugates.

\begin{proposition}
\label{p:170513a}
Let $\theta\colon X\to\left]-\infty,+\infty\right]$
be such that $U\cap \dom\theta\neq\varnothing$
and let $\gamma\in\RPP$.
Then the following hold\footnote{
Indeed, the proof does not require any of
\textbf{A1}--\textbf{A4}.
}:
\begin{enumerate}
\item
\label{p:170513ai}
$\gamma\benv[\gamma]{\theta} \circ\nabla f^* = 
f^* - (\gamma\theta+f)^*$.
\item
$\gamma\fenv[\gamma]{\theta} = f - (f^*+(\gamma\theta\circ\nabla
f))^*$. 
\label{p:170513aii}
\end{enumerate}
\end{proposition}
\begin{proof}
\ref{p:170513ai}: This follows from \cite[Theorem~2.4]{KS} and \eqref{e:Leggy}.
Note also that the case in which $\gamma=1$ is related to \cite[Theorem~1(i)]{CR13} applied to $(f^*, \theta^*)$ instead of $(f, \theta)$.
 
\ref{p:170513aii}:
Let $x\in X$.
Using the fact that $f^*(\nabla f(y)) =\scal{\nabla f(y)}{y} -f(y)$ (see, e.g., \cite[Theorem~23.5]{Roc70})
and that $\big(\nabla f\big)^{-1} =\nabla f^*$ (see \eqref{e:Leggy}), we obtain
\begin{subequations}
\begin{align}
\fenv[\gamma]{\theta}(x) 
&= \inf_{y\in X} \Big( \theta(y)+\frac{1}{\gamma}
\big(f(x)-f(y)-\scal{\nabla f(y)}{x-y}\big) \Big)\\
&= \frac{f(x)}{\gamma} + 
\frac{1}{\gamma}\inf_{y\in U}\big(\gamma\theta(y) + 
f^*(\nabla f(y))-\scal{\nabla f(y)}{x}\big)\\
&= \frac{f(x)}{\gamma} + 
\frac{1}{\gamma}\inf_{y^*\in U^*}\big(\gamma\theta(\nabla
f^*(y^*))+ 
f^*(y^*)-\scal{y^*}{x}\big)\\
&= \frac{f(x)}{\gamma} - 
\frac{1}{\gamma}\sup_{y^*\in X}
\big(\scal{x}{y^*} - \big((\gamma\theta\circ\nabla
f^*)+f^*\big)(y^*) \big)\\
&= \frac{f(x)}{\gamma} - 
\frac{1}{\gamma}
\big((\gamma\theta\circ\nabla f^*)+f^*\big)^*(x).
\end{align}
\end{subequations}
This completes the proof. 
\end{proof}

In what follows, we shall require the following two facts. 

\begin{fact} 
\label{f:marco}
The following hold:
\begin{enumerate} 
\item \label{f:marcoi}
$(\forall x\in X)(\forall y\in U)\;\;D_f(x,y)=0\;\;\Leftrightarrow\;\;x=y$.
\item \label{f:marcoii}
$(\forall y\in U)$ $D_f(\cdot,y)$ is coercive, 
i.e., $D_f(x, y) \to +\infty$ as $\|x\| \to +\infty$.
\end{enumerate}
\end{fact}
\begin{proof}
\ref{f:marcoi}: See \cite[Theorem~3.7.(iv)]{BB97}.
\ref{f:marcoii}: See \cite[Theorem~3.7.(iii)]{BB97}.
\end{proof}

\begin{fact} 
\label{f:coeR}
Let $\theta\in\Gamma_0(X)$ be such that $\dom\theta\cap U\neq\varnothing$ and let $\gamma \in \RPP$.
Consider the following properties:
\begin{itemize}
\item[\rm (a)]
$U\cap \dom\theta$ is bounded.
\item[\rm (b)]
$\inf\theta(U)>-\infty$.
\item[\rm (c)]
$f$ is supercoercive, i.e., 
$f(x)/\|x\| \to +\infty$ as $\|x\| \to +\infty$.
\item[\rm (d)]
$(\forall x\in U)$ $D_f(x,\cdot)$ is supercoercive.
\end{itemize}
Then the following hold:
\begin{enumerate}
\item \label{l:coeRi}
If any of the conditions (a), (b), or (c) holds, then 
\begin{equation} \label{e:bcoeR}
(\forall y \in U) \quad \theta(\cdot) + \frac{1}{\gamma}D_f(\cdot,y) \;\;
\text{is coercive}
\end{equation}
or, equivalently,
\begin{equation}
\frac{1}{\gamma}\ran\nabla f \subseteq \intdom\left(\frac{1}{\gamma}f +\theta\right)^*.
\label{e:brilliantissimo}
\end{equation}

\item \label{l:coeRii}
If any of the conditions (a), (b), or (d) holds, then 
\begin{equation} \label{e:fcoeR}
(\forall x \in U) \quad \theta(\cdot) + \frac{1}{\gamma}D_f(x,\cdot) \;\;
\text{is coercive.}
\end{equation}
\end{enumerate}
\end{fact}
\begin{proof}
Since $\frac{1}{\gamma}D_f =D_{\frac{1}{\gamma}f}$,
the result follows from \cite[Lemma~2.12]{BCN06} 
applied to $\tfrac{1}{\gamma}f$.
\end{proof}

The definition of proximal mappings relies on the following
result. 
(For variants of \cref{p:exact}\ref{p:exact_left},
see \cite[Theorems~2.2 and 4.3]{KS}.)

\begin{proposition} 
\label{p:exact} 
Let $\theta \colon X \to \left]-\infty, +\infty\right]$ be convex
and such that $U\cap \dom\theta\neq\varnothing$, and let $\gamma \in \RPP$.
Then the following hold:
\begin{enumerate}
\item 
\label{p:exact_left}
$\benv[\gamma]{\theta}$ is convex and continuous on $U$, and
	\begin{enumerate}
	\item 
	if \eqref{e:bcoeR} holds, i.e., $(\forall y \in U)$ $\theta(\cdot) +\frac{1}{\gamma}D_f(\cdot,y)$ is coercive, 
	then $\benv[\gamma]{\theta}$ is proper; 
	\item
	if $\theta\in\Gamma_0(X)$ and $\theta(\cdot) +\frac{1}{\gamma}D_f(\cdot,y)$ is coercive for a given $y\in U$, 
	then there exists a unique point $z\in U$ such that $\benv[\gamma]{\theta}(y)=\theta(z) +\frac{1}{\gamma}D_f(z,y)$.
	\end{enumerate}

\item 
\label{p:exact_right} 
$\fenv[\gamma]{\theta}$ is convex and continuous on $U$, and
	\begin{enumerate}
	\item 
	if \eqref{e:fcoeR} holds, i.e., $(\forall x \in U)$ $\theta(\cdot) +\frac{1}{\gamma}D_f(x,\cdot)$ is coercive, 
	then $\fenv[\gamma]{\theta}$ is proper; 
	\item
	if $\theta\in\Gamma_0(X)$ and $\theta(\cdot) +\frac{1}{\gamma}D_f(x,\cdot)$ is coercive for a given $x\in U$, 
	then there exists a unique point $z\in U$ such that $\fenv[\gamma]{\theta}(z)=\theta(z) +\frac{1}{\gamma}D_f(x,z)$.
	\end{enumerate}
\end{enumerate}
\end{proposition}
\begin{proof}
Since $\frac{1}{\gamma}D_f =D_{\frac{1}{\gamma}f}$,
the result follows from
 \cite[Propositions~3.4 and 3.5]{BCN06}
applied to $\tfrac{1}{\gamma}f$.
\end{proof}

In view of \cref{p:exact}, we define the following operators on $U$; 
see also \cite[Definition~3.7]{BCN06}.

\begin{definition}[Bregman proximity operators] 
\label{d:prox}
Let $\theta\in\Gamma_0(X)$ be such that 
$U\cap \dom\theta\neq\varnothing$.
If \eqref{e:bcoeR} holds for $\gamma =1$, then 
the \emph{left proximity operator} associated with $\theta$ is
\begin{equation} \label{e:dbprox}
\bprox{\theta} \colon U \to U \colon
y\mapsto
\underset{x\in X}{\operatorname{argmin}}\;\: \big(\theta(x) +D_f(x,y)\big).
\end{equation}
If \eqref{e:fcoeR} holds for $\gamma =1$, then 
the \emph{right proximity operator} associated with $\theta$ is
\begin{equation} \label{e:dfprox}
\fprox{\theta}\colon U \to  U \colon
x \mapsto
\underset{y\in X}{\operatorname{argmin}}\;\: \big(\theta(y) +D_f(x,y)\big).
\end{equation}
\end{definition}

\begin{remark}
\label{r:energy}
Suppose that $f =\frac{1}{2}\|\cdot\|^2$ and let $\theta\in\Gamma_0(X)$. 
Then $U =\intdom f =X$ and hence 
$U\cap \dom\theta =\dom\theta \neq\varnothing$.
Since $f(x)/\|x\| =\frac{1}{2}\|x\| \to +\infty$ as $\|x\| \to +\infty$, \cref{f:coeR} implies that
\eqref{e:bcoeR} and \eqref{e:fcoeR} hold for all $\gamma \in \RPP$. 
In this case, $D_f\colon (x, y)\mapsto \frac{1}{2}\|x -y\|^2$ 
and $\bprox{\theta} =\fprox{\theta} =\prox_\theta$ is the classical \emph{Moreau proximity operator} of $\theta$ \cite{Mor62}.
\end{remark}

Given a closed convex subset $C$ of $X$ with $C\cap U \neq\varnothing$, we have that $\iota_C \in \Gamma_0(X)$, $\dom\iota_C =C$, 
and hence $U\cap \dom\iota_C =U\cap C \neq\varnothing$ and also $\inf \iota_C(U) =0 >-\infty$, 
which together with \cref{f:coeR} imply that \eqref{e:bcoeR} and \eqref{e:fcoeR} hold for all $\gamma \in \RPP$. 
This leads to the following definition. 

\begin{definition}[Bregman projectors]
\label{d:BregProj}
Let $C$ be a closed convex subset of $X$ such that 
$U\cap C\neq\varnothing$. 
Then $\bproj{C} :=\bprox{\iota_C}$ is the \emph{left Bregman projector} onto $C$
and $\fproj{C} :=\fprox{\iota_C}$ is the \emph{right Bregman projector} onto $C$.
\end{definition}

\begin{remark}
\label{r:BregProj}
In view of \cref{r:energy}, if $f =\frac{1}{2}\|\cdot\|^2$, 
then $\bproj{C} =\fproj{C} =\proj{C}$ is the \emph{orthogonal projector} onto $C$. 
Note that $\bproj{C}$, $\fproj{C}$, and $\proj{C}$ are not, in general, the same when $f \ne \frac12 \|\cdot\|^2$.
Before we give a corresponding example, let us 
show that these projectors are the same when $X=\RR$.
\end{remark}

\begin{proposition}
\label{p:allproj}
Suppose that $X =\RR$ and let $C$ be a closed convex subset of $\RR$
such that $U\cap C \neq\varnothing$. 
Then $\bproj{C} =\fproj{C} =\proj{C}$ on $U$.
\end{proposition}
\begin{proof}
Let $y \in U$. 
Because $X =\RR$, $(\forall z \in C)$
$(\exists\, \lambda_z \in \left[0, 1\right])$ 
$\proj{C}y =\lambda_z z +(1 -\lambda_z)y$.
Since $D_f(\cdot, y)$ is convex, nonnegative, and $D_f(y, y) =0$,
it follows that 
\begin{equation}
(\forall z \in C)\quad D_f(\proj{C}y, y) \leq \lambda_z D_f(z, y) +(1 -\lambda_z)D_f(y, y) 
=\lambda_z D_f(z, y) \leq D_f(z, y). 
\end{equation}
This combined with \cref{d:BregProj} yields
$\bproj{C}(y) =\proj{C}(y)$.
The proof that $\fproj{C} =\proj{C}$ is similar.
\end{proof}

\begin{example}
\label{ex:proj}
Here we illustrate how Bregman projectors may differ from the orthogonal
projector. We adapt \cite[Example~6.15]{BB97}, which illustrates 
the setting in which $f$ is an entropy function on $\RR^J$ and $C$ is
the ``probabilistic hyperplane'' $\menge{x\in \RR^J}{\sum_j \xi_j=1}$.
For simplicity, we work in $X=\RR^2$. 
Suppose that $f_1$ is the energy from
Example~\ref{ex:examples}(i) while
$f_2$ is the negative Boltzmann--Shannon entropy 
from Example~\ref{ex:examples}(ii). 
	Since we work in $\RR^2$, the probabilistic hyperplane
	is described by $\xi_2=1-\xi_1$. 
We compute $\bproj{C}(1,0)$ by substituting
$\eta_1=1,\eta_2=0,\xi_2=1-\xi_1$ and minimizing the resulting
Bregman distance over $\xi_1$. We obtain 
\begin{enumerate}
\item 
$\bproj{C}(1,2)=(0,1)$ for 
$D_{f_1}$,
\item 
$\bproj{C}(1,2)=\left( 1/3,2/3\right)$ for 
$D_{f_2}$. 
\end{enumerate}
	We illustrate this in Figure~\ref{fig:proj}. For
	$i\in\{1,2\}$,
	we sketch the contour plot of $D_{f_i}(\cdot,(1,2))$ for
	the level given by  $D_{f_i}(\bproj{C}(1,2),(1,2))$ together with the set $C$.
\end{example}

\begin{figure}
	\begin{center}
		\includegraphics[width=.5\columnwidth]{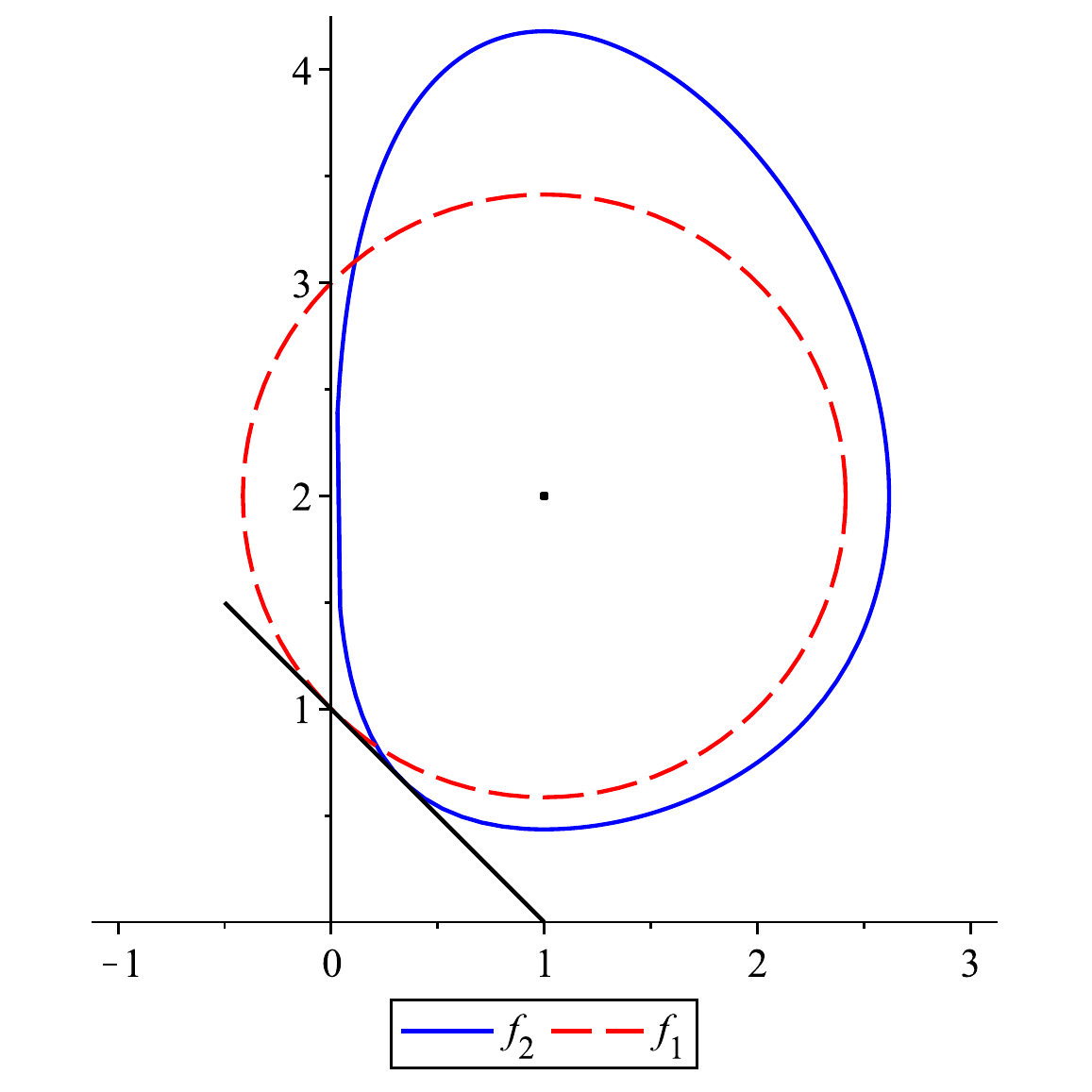}
	\end{center}
\caption{\cref{ex:proj} is illustrated}
\label{fig:proj}
\end{figure}

\begin{remark}
\label{r:prox}
Let $\theta\in\Gamma_0(X)$ be such that 
$U\cap \dom\theta\neq\varnothing$ and let $\gamma \in \RPP$.
\cref{p:12.22i} implies that 
$\benv{\gamma\theta} =\gamma\benv[\gamma]{\theta}$ 
and $\fenv{\gamma\theta} =\gamma\fenv[\gamma]{\theta}$. 
We thus derive from the definition that if \eqref{e:bcoeR} holds, then
\begin{subequations}
\label{e:benv-bproxall}
\begin{equation}
\label{e:benv-bprox}
\benv[\gamma]{\theta}(y) =\theta\big(\bprox{\gamma\theta}(y)\big) 
+\frac{1}{\gamma}D_f\big(\bprox{\gamma\theta}(y), y\big)
\end{equation}
and, by combining with \cref{p:12.9}\ref{p:12.9_left},
\begin{equation}
\label{e:benv-bprox'}
\theta\big(\bprox{\gamma\theta}(y)\big) \leq \benv[\gamma]{\theta}(y) \leq \theta(y).
\end{equation}
\end{subequations}
Similarly, if \eqref{e:fcoeR} holds, then
\begin{subequations}
\label{e:fenv-fproxall}
\begin{equation}
\label{e:fenv-fprox}
\fenv[\gamma]{\theta}(x) =\theta\big(\fprox{\gamma\theta}(x)\big) +
\frac{1}{\gamma}D_f\big(x, \fprox{\gamma\theta}(x)\big)
\end{equation}
and 
\begin{equation}
\theta\big(\fprox{\gamma\theta}(x)\big) \leq \fenv[\gamma]{\theta}(x) \leq \theta(x).
\end{equation}
\end{subequations}
\end{remark}

The next result provides information on the proximal mapping when
the parameter is varied. 
For a variant of the last inequality in \eqref{e:170513b}, 
see \cite[Proposition~2.1(ii)]{KS}. 

\begin{proposition}
\label{p:Lange}
Let $\theta\in\Gamma_0(X)$ be such that 
$U\cap \dom\theta\neq\varnothing$ and let $\gamma \in \RPP$.
\begin{enumerate}
\item 
\label{p:Langei}
If \eqref{e:bcoeR} holds, then
$(\forall y\in U)(\forall\mu \in \left]\gamma, +\infty\right[)$
\begin{equation}
\label{e:170513b}
\theta\big(\bprox{\mu\theta}(y)\big) \leq
\theta\big(\bprox{\gamma\theta}(y)\big)
\quad\text{and}\quad D_f\big(\bprox{\mu\theta}(y), y\big) \geq
D_f\big(\bprox{\gamma\theta}(y), y\big).
\end{equation}
\item 
\label{p:Langeii}
If \eqref{e:fcoeR} holds, then $(\forall x\in U)(\forall\mu \in \left]\gamma, +\infty\right[)$
\begin{equation}
\theta\big(\fprox{\mu\theta}(x)\big) \leq
\theta\big(\bprox{\gamma\theta}(x)\big)
\quad\text{and}\quad D_f\big(\bprox{\mu\theta}(x), x\big) 
\geq D_f\big(\bprox{\gamma\theta}(x), x\big).
\end{equation}
\end{enumerate}
\end{proposition}
\begin{proof}
This follows from \cref{r:prox} and \cite[Proposition~7.6.1]{Lan16}.
\end{proof}

The left and right proximal mappings can be characterized in
various ways:

\begin{proposition}
\label{p:prox}
Let $\theta\in\Gamma_0(X)$ be such that $\dom\theta\cap U\neq\varnothing$ and let $\gamma \in \RPP$.
\begin{enumerate}
\item 
\label{p:prox_left}
Suppose that \eqref{e:bcoeR} holds.
Then for every $(x,y) \in U\times U$, 
the following conditions are equivalent:
	\begin{enumerate}
	\item 
	$x =\bprox{\gamma\theta}(y)$,
	\item 
	$0 \in \gamma\partial\theta(x) +\nabla f(x) -\nabla f(y)$,
	\item 
	$(\forall z \in X)\quad \scal{\nabla f(y) -\nabla f(x)}{z -x} +\gamma\theta(x) \leq \gamma\theta(z)$.
	\end{enumerate}
Moreover, 
\begin{equation} \label{e:pro}
\bprox{\gamma\theta}=(\nabla f +\gamma\partial\theta)^{-1}\circ \nabla f 
\end{equation}
is continuous on $U$. 
\item 
\label{p:prox_right}
Suppose that \eqref{e:fcoeR} holds. 
Then for every $(x,y) \in U\times U$, 
the following conditions are equivalent:
	\begin{enumerate}
	\item 
	$y =\fprox{\gamma\theta}(x)$,
	\item 
	$0 \in \gamma\partial\theta(y) +\nabla^2 f(y)(y -x)$,
	\item 
	$(\forall z \in X)\quad \scal{\nabla^2 f(y)(x -y)}{z -y} +\gamma\theta(y) \leq \gamma\theta(z)$.
	\end{enumerate}
Moreover, $\fprox{\gamma\theta}$ is continuous on $U$. 
\end{enumerate}
\end{proposition}
\begin{proof}
Apply \cite[Proposition~3.10]{BCN06} to $\gamma\theta$.
\end{proof}

\begin{remark}
Consider \cref{p:prox} and its notation.
\begin{enumerate}
\item
In the case of item~\ref{p:prox_left} and when $U^*=X$, we note that, 
by \eqref{e:Leggy}, 
\begin{equation}
\bprox{\gamma\theta} \circ\nabla f^* = \big(\nabla f
+\gamma\partial\theta\big)^{-1} = \partial\big(f+\gamma\theta\big)^*
\quad\text{is
maximally (cyclically) monotone};
\end{equation}
see also \cite[Theorem~4.2]{KS} for a more general result. 
\item
In the case of item~\ref{p:prox_right}, let us prove
the variant of \cite[Theorem~4.1]{KS}
stating that 
\begin{equation}
\label{e:tada}
\nabla f\circ \fprox{\gamma\theta}\quad \text{is monotone.}
\end{equation}
Indeed, let $x_1$ and $x_2$ be in $U$,
and set $y_i = \fprox{\gamma\theta}(x_i)$ for $i\in\{1,2\}$.
Then
$\theta(y_1)+\tfrac{1}{\gamma}D_f(x_1,y_1)
\leq \theta(y_2)+\tfrac{1}{\gamma}D_f(x_1,y_2)$ and
$\theta(y_2)+\tfrac{1}{\gamma}D_f(x_2,y_2)
\leq \theta(y_1)+\tfrac{1}{\gamma}D_f(x_2,y_1)$.
Adding and simplifying yields
\begin{equation}
\label{e:4points}
0 \leq 
D_f(x_1,y_2)+D_f(x_2,y_1)-D_f(x_1,y_1)-D_f(x_2,y_2).
\end{equation}
A direct expansion (or the \emph{four-point identity} from
\cite[Remark~2.5]{BauLew}) shows that 
\eqref{e:4points} is the same as 
\begin{equation}
0 \leq \scal{\nabla f(y_1)-\nabla f(y_2)}{x_1-x_2};
\end{equation}
therefore, \eqref{e:tada} follows.
We do not know whether or not in general the operator in \eqref{e:tada} 
is the gradient of a convex function. 
\end{enumerate}
\end{remark}

\begin{corollary}
\label{c:BregProj}
Let $C$ be a closed convex subset of $X$ such that 
$U\cap C\neq\varnothing$, 
let $(x,y) \in U\times U$, 
and let $p \in U\cap C$.
Then the following hold:
\begin{enumerate}
\item 
\label{c:BregProj_left}
$p =\bproj{C}y \ \Leftrightarrow \ (\forall z \in C)\; 
\scal{\nabla f(y) -\nabla f(p)}{z -p} \leq 0$. 
\item 
\label{c:BregProj_right}
$p =\fproj{C}x \ \Leftrightarrow \ (\forall z \in C)\; 
\scal{\nabla^2 f(p)(x -p)}{z -p} \leq 0$. 
\end{enumerate}
\end{corollary}
\begin{proof}
In light of \cref{d:BregProj}, we 
apply \cref{p:prox} 
(see also \cite[Proposition~3.16]{BB97}).
\end{proof}

The derivatives of the left and right Bregman--Moreau envelopes
feature the corresponding proximal mappings as follows.

\begin{proposition} 
\label{p:grad}
Let $\theta\in\Gamma_0(X)$ be such that 
$U\cap \dom\theta\neq\varnothing$ and let $\gamma \in \RPP$.
Then the following hold:
\begin{enumerate}
\item
\label{p:grad_left}
If \eqref{e:bcoeR} holds, then $\benv[\gamma]{\theta}$ is differentiable on $U$ and 
\begin{equation} 
(\forall y\in U)\quad \nabla\benv[\gamma]{\theta}(y) =\frac{1}{\gamma}\nabla^2 f(y)(y -\bprox{\gamma\theta}(y)).
\end{equation}
\item 
\label{p:grad_right}
If \eqref{e:fcoeR} holds, then $\fenv[\gamma]{\theta}$ is differentiable on $U$ and 
\begin{equation} 
(\forall x\in U)\quad \nabla\fenv[\gamma]{\theta}(x) =\frac{1}{\gamma}\nabla f(x) -\frac{1}{\gamma}\nabla f(\fprox{\gamma\theta}(x)).
\end{equation}
\end{enumerate}
\end{proposition}
\begin{proof}
Combine Remark~\ref{r:prox}
with \cite[Proposition~3.12]{BCN06}.
\end{proof}

The following result, which is a variant of
\cite[Theorem~XV.4.1.7]{HUL93}, highlights the connection to
convex optimization. 

\begin{theorem}
\label{p:min}
Let $\theta\in\Gamma_0(X)$ be such that 
$U\cap \dom\theta\neq\varnothing$, let $\gamma \in \RPP$, and let $x, y \in U$.
\begin{enumerate}
\item 
\label{p:min_left}
Suppose that \eqref{e:bcoeR} holds. Then the following are equivalent:
	\begin{enumerate}
	\item \label{p:min_left1} $y \in \argmin\theta$,
	\item \label{p:min_left2} $y \in \Fix\bprox{\gamma\theta}$,
	\item \label{p:min_left3} $y \in \argmin\benv[\gamma]{\theta}$,
	\item \label{p:min_left4} $\theta(\bprox{\gamma\theta}(y)) =\theta(y)$,
	\item \label{p:min_left5} $\benv[\gamma]{\theta}(y) =\theta(y)$.
	\end{enumerate} 
Consequently, 
\begin{equation}
\label{e:min_left}
U\cap \argmin\theta =\Fix\bprox{\gamma\theta}
=\argmin\benv[\gamma]{\theta}.
\end{equation}
\item
\label{p:min_right}
Suppose that \eqref{e:fcoeR} holds. Then the following are equivalent:
	\begin{enumerate}
	\item \label{p:min_right1} $x \in \argmin\theta$,
	\item \label{p:min_right2} $x \in \Fix\fprox{\gamma\theta}$,
	\item \label{p:min_right3} $x \in \argmin\fenv[\gamma]{\theta}$,
	\item \label{p:min_right4} $\theta(\fprox{\gamma\theta}(x)) =\theta(x)$,
	\item \label{p:min_right5} $\fenv[\gamma]{\theta}(x) =\theta(x)$.
	\end{enumerate} 
Consequently, 
\begin{equation}
\label{e:min_right}
U\cap \argmin\theta =\Fix\fprox{\gamma\theta} =U\cap \argmin\fenv[\gamma]{\theta}.
\end{equation}
\end{enumerate}
\end{theorem}
\begin{proof}
\ref{p:min_left}: Using \cref{p:prox}\ref{p:prox_left}, we have
\begin{subequations}
\begin{align}
y \in \argmin\theta &\Leftrightarrow \ 0 \in \partial\theta(y) \Leftrightarrow \ 0 \in \gamma\partial\theta(y) +\nabla f(y) -\nabla f(y) \\
&\Leftrightarrow \ y =\bprox{\gamma\theta}(y) \Leftrightarrow \ y \in \Fix\bprox{\gamma\theta}.
\end{align}
\end{subequations}
This proves that \ref{p:min_left1} $\Leftrightarrow$ \ref{p:min_left2}. 

Assume that \ref{p:min_left2} holds, i.e., $y =\bprox{\gamma\theta}(y)$. 
Then 
$\nabla\benv[\gamma]{\theta}(y) =0$ by \cref{p:grad}\ref{p:grad_left}, 
and thus \ref{p:min_left3} holds by the convexity of $\benv[\gamma]{\theta}$ shown in \cref{p:exact}\ref{p:exact_left}.
Next, \ref{p:min_left4} is obvious and \ref{p:min_left5} holds due to \eqref{e:benv-bprox}.

Now recall from \eqref{e:benv-bproxall} that 
\begin{equation}
\label{e:benv-bprox''}
\theta\big(\bprox{\gamma\theta}(y)\big) \leq 
\theta\big(\bprox{\gamma\theta}(y)\big) +\frac{1}{\gamma}
D_f\big(\bprox{\gamma\theta}(y), y\big) 
=\benv[\gamma]{\theta}(y) \leq \theta(y).
\end{equation}
If \ref{p:min_left3} holds, then since $\inf\benv[\gamma]{\theta}(X) \leq \inf\theta(U)$ (see \cref{p:12.9}\ref{p:12.9_left}), 
combining with \eqref{e:benv-bprox''} yields
\begin{equation}
\inf\theta(U) \leq \theta(\bprox{\gamma\theta}(y)) \leq \benv[\gamma]{\theta}(y) =\min\benv[\gamma]{\theta}(X) \leq \inf\theta(U),
\end{equation}
which implies that $D_f(\bprox{\gamma\theta}(y), y) =0$, so $y =\bprox{\gamma\theta}(y)$ due to \cref{f:marco}\ref{f:marcoi}, 
and we get \ref{p:min_left2}.

If \ref{p:min_left4} holds, then by \eqref{e:benv-bprox''}, $D_f(\bprox{\gamma\theta}(y), y) =0$, and \ref{p:min_left2} thus holds.

Finally, if \ref{p:min_left5} holds, then $\benv[\gamma]{\theta}(y) =\theta(y) +\frac{1}{\gamma}D_f(y, y)$, 
and using \cref{p:exact}\ref{p:exact_left} and \eqref{e:benv-bprox}, we must have $y =\bprox{\gamma\theta}(y)$ 
and therefore get \ref{p:min_left2}.  

\ref{p:min_right}: This is proved similarly to \ref{p:min_left} by using \cref{p:12.9}\ref{p:12.9_right}, 
\cref{p:exact}\ref{p:exact_right}, \cref{p:prox}\ref{p:prox_right}, \cref{p:grad}\ref{p:grad_right}, and \eqref{e:fenv-fprox}. 
The difference between 
\eqref{e:min_left} and \eqref{e:min_right} is because 
$\dom\benv[\gamma]{\theta} =U$ while $\dom\fenv[\gamma]{\theta} =\dom f$.
\end{proof}

\section{Asymptotic behaviour properties}

\label{s:drei}

The results in this section, almost all of which are new,
extend or complement results for the classical energy case and
for left variants studied in \cite{CKS} and \cite{KS}. 
We will require the following lemma. 

\begin{lemma}
\label{l:cvg}
Let $C$ be a compact subset of a Hausdorff space $\mathcal X$,
let $\phi\colon \mathcal X \to \left[-\infty, +\infty\right]$ be lower 
semicontinuous,
let $(x_a)_{a\in A}$ be a net in $C$,
and suppose that $\phi(x_a) \to \inf\phi(\mathcal X)$. 
Then $\argmin\phi \neq\varnothing$ 
and all cluster points of $(x_a)_{a\in A}$ lie in $\argmin\phi$.
Consequently, if $\phi$ attains its minimum at a unique point $u$, then $x_a \to u$.
\end{lemma}
\begin{proof}
This follows from the lower semicontinuity of $\phi$ and \cite[Lemma 1.14]{BC17}.
\end{proof}

What is the behaviour of Bregman--Moreau envelopes and proximity operators when $\gamma\downarrow 0$?
The next two results provide answers.

\begin{proposition}
\label{p:limprox}
Let $\theta\in\Gamma_0(X)$ be such that 
$U\cap \dom\theta\neq\varnothing$ and let $x, y \in U$.
Then the following hold:
\begin{enumerate}
\item 
\label{p:limprox_left}
If \eqref{e:bcoeR} holds for some $\mu \in \RPP$ instead of $\gamma$, 
then $\bprox{\gamma\theta}(y) \to y$ as $\gamma \downarrow 0$.
\item 
\label{p:limprox_right}
If \eqref{e:fcoeR} holds for some $\mu \in \RPP$ instead of $\gamma$, 
then $\fprox{\gamma\theta}(x) \to x$ as $\gamma \downarrow 0$.
\end{enumerate}	
\end{proposition}
\begin{proof}
\ref{p:limprox_left}: Noting that $(\forall \gamma \in \left]0, \mu\right])$ 
$\theta +\frac{1}{\gamma}D_f(\cdot, y) \geq \theta +\frac{1}{\mu}D_f(\cdot, y)$,
we have that \eqref{e:bcoeR} holds for all $\gamma \in \left]0, \mu\right]$.
In particular, $g :=\theta +\tfrac{1}{\mu}D_f(\cdot, y)$ is coercive.
By \cref{p:12.9}\ref{p:12.9_left} and \eqref{e:benv-bprox}, 
\begin{equation}
\label{e:gamma<1}
\big(\forall\gamma \in \left]0, \mu\right]\big)\quad 
\theta(y) \geq \benv[\gamma]{\theta}(y)
=\theta\big(\bprox{\gamma\theta}(y)\big) +\frac{1}{\gamma}
D_f\big(\bprox{\gamma\theta}(y), y\big) 
\geq g\big(\bprox{\gamma\theta}(y)\big)
\end{equation}
and so $\bprox{\gamma\theta}(y) \in \lev{\theta(y)}g$. 
The coercivity of $g$ and \cite[Proposition~11.12]{BC17} imply that 
$\nu :=\sup_{\gamma \in \left]0, \mu\right]} \|\bprox{\gamma\theta}(y)\| <+\infty$.
Now by \cite[Theorem~9.20]{BC17}, 
there exist $u \in X$ and $\eta \in \RR$ such that 
$\theta \geq \scal{\cdot}{u} +\eta$. 
Using \eqref{e:gamma<1} and Cauchy--Schwarz yields
\begin{subequations}
\begin{align}
\big(\forall\gamma \in \left]0, \mu\right]\big)\quad 
\theta(y) 
&\geq \theta\big(\bprox{\gamma\theta}(y)\big) +\frac{1}{\gamma}
D_f\big(\bprox{\gamma\theta}(y), y\big)\\
&\geq \scal{\bprox{\gamma\theta}(y)}{u} +\eta
+\frac{1}{\gamma}D_f\big(\bprox{\gamma\theta}(y), y\big) \\
&\geq -\nu\|u\| +\eta
+\frac{1}{\gamma}D_f\big(\bprox{\gamma\theta}(y), y\big), 
\end{align}
\end{subequations}
which gives
\begin{equation}
0 \leq D_f\big(\bprox{\gamma\theta}(y), y\big) 
\leq \gamma\big(\theta(y) +\nu\|u\| -\eta\big) \to 0 \quad\text{as}\quad \gamma \downarrow 0,
\end{equation}
and thus $D_f(\bprox{\gamma\theta}(y), y) \to 0$ as $\gamma \downarrow 0$. 
Observing that $D_f(\cdot, y) =f(\cdot) -f(y) -\scal{\nabla f(y)}{\cdot -y}$ is lower semicontinuous, 
that $\argmin D_f(\cdot, y) =\{y\}$ by \cref{f:marco}\ref{f:marcoi}, 
and that $\sup_{\gamma \in \left]0, 1\right[} \|\bprox{\gamma\theta}(y)\| <+\infty$, 
it follows from \cref{l:cvg} that $\bprox{\gamma\theta}(y) \to y$ as $\gamma \downarrow 0$.

\ref{p:limprox_right}: This is similar to \ref{p:limprox_left}. 
\end{proof}

\begin{theorem}
\label{p:12.32}
Let $\theta\in\Gamma_0(X)$ be such that 
$U\cap \dom\theta\neq\varnothing$ and let $x, y \in U$.
Then the following hold:
\begin{enumerate}
\item 
\label{p:12.32_left}
If \eqref{e:bcoeR} holds for some $\mu \in \RPP$ 
and $\gamma \downarrow 0$, then
$\benv[\gamma]{\theta}(y) \uparrow \theta(y)$,
$\theta(\bprox{\gamma\theta}(y))\uparrow \theta(y)$,
and 
$\tfrac{1}{\gamma}D_f(\bprox{\gamma\theta}(y), y)\to 0$.
\item
\label{p:12.32_right}
If \eqref{e:fcoeR} holds for some $\mu \in \RPP$,
and $\gamma \downarrow 0$, then
$\fenv[\gamma]{\theta}(x) \uparrow \theta(x)$,
$\theta(\fprox{\gamma\theta}(x))\uparrow \theta(x)$,
and 
$\tfrac{1}{\gamma}D_f(x,\fprox{\gamma\theta}(x))\to 0$.
\end{enumerate}
\end{theorem}
\begin{proof}
\ref{p:12.32_left}: According to \cref{p:12.9}\ref{p:12.9_left}, 
there exists $\beta \in \RR$ such that $\benv[\gamma]{\theta}(y) \uparrow \beta \leq \theta(y)$ as $\gamma \downarrow 0$.
Combining with \eqref{e:benv-bprox}, we have
\begin{equation}
\label{e:170430x}
\big(\forall\gamma \in \left]0, \mu\right]\big)\quad 
\theta(y) \geq \beta \geq \benv[\gamma]{\theta}(y) 
=\theta\big(\bprox{\gamma\theta}(y)\big) +\frac{1}{\gamma}
D_f\big(\bprox{\gamma\theta}(y), y\big) 
\geq \theta\big(\bprox{\gamma\theta}(y)\big).
\end{equation}
This together with the fact that $\lim_{\gamma \downarrow 0} \bprox{\gamma\theta}(y) =y$ by \cref{p:limprox}\ref{p:limprox_left}, 
and the lower semicontinuity of $\theta$ implies 
\begin{equation}
\theta(y) \geq \beta \geq \varliminf_{\gamma \downarrow 0}
\theta\big(\bprox{\gamma\theta}(y)\big) \geq \theta(y) \geq 
 \varlimsup_{\gamma \downarrow 0}
\theta\big(\bprox{\gamma\theta}(y)\big),
\end{equation} 
and then $\beta =\theta(y)=\lim_{\gamma\downarrow 0}\theta(\bprox{\gamma\theta}(y))$.
Now recall \eqref{e:170430x} and Proposition~\ref{p:Lange}\ref{p:Langei}. 

\ref{p:12.32_right}: This is similar to \ref{p:12.32_left}.
\end{proof}

For a variant of the result from \cref{p:12.32}\ref{p:12.32_left} that 
$\benv[\gamma]{\theta}(y) \uparrow \theta(y)$ as $\gamma \downarrow 0$, see \cite[Theorem~2.5]{KS}.
Note that $D_f(\bprox{\gamma\theta}(y), y)$ is monotone with respect to $\gamma$, 
as shown in \cref{p:Lange}\ref{p:Langei}, 
but the same is not necessarily true for $\tfrac{1}{\gamma}D_f(\bprox{\gamma\theta}(y), y)$ (see \cref{fig:Dh}).
\begin{figure}
	\begin{center}
		\vspace*{\fill}
		\adjustbox{trim={.05\width} {0\height} {0.05\width} {0\height},clip}%
		{\includegraphics[width=.36\linewidth]{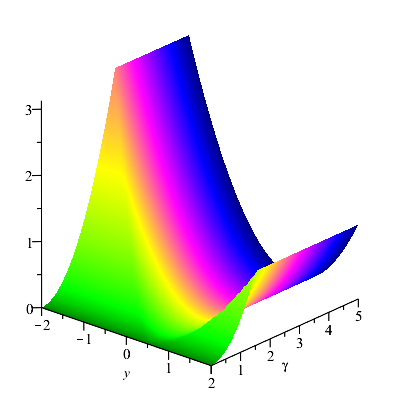}}
		\vspace*{\fill}
		\includegraphics[width=.1\linewidth]{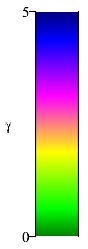}
		\vspace*{\fill}
		\adjustbox{trim={.0\width} {0\height} {0.05\width} {0\height},clip}%
		{\includegraphics[width=.4\linewidth]{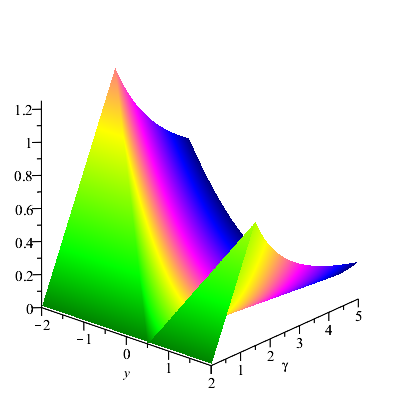}}
		\vspace*{\fill}
	\end{center}

\caption{$D_f\big(\protect\bprox{\gamma\theta}(y), y\big)$ (left) 
and $\frac{1}{\gamma}D_f\big(\protect\bprox{\gamma\theta}(y), y\big)$ (right) 
when $X =\RR$, $f$ is the energy, and $\theta$ is the function $x \mapsto |x -\tfrac{1}{2}|$}
\label{fig:Dh}
\end{figure}

The two following results describe the 
behaviour when $\gamma\uparrow +\infty$.

\begin{proposition}
\label{p:limthetaprox}
Let $\theta\in\Gamma_0(X)$ be such that $U\cap \dom\theta\neq\varnothing$ and let $x, y \in U$.
Then the following hold:
\begin{enumerate}
\item 
\label{p:limthetaprox_left}
If \eqref{e:bcoeR} holds for all $\gamma\in\RPP$, 
then $\theta(\bprox{\gamma\theta}(y)) \to \inf \theta(X)$ as $\gamma \uparrow +\infty$.
\item 
\label{p:limthetaprox_right}
If \eqref{e:fcoeR} holds for all $\gamma\in \RPP$, then  
$\theta(\fprox{\gamma\theta}(x)) \to \inf \theta(X)$ as $\gamma \uparrow +\infty$.
\end{enumerate}	
\end{proposition}
\begin{proof}
We shall just prove \ref{p:limthetaprox_left} 
because the proof of \ref{p:limthetaprox_right} is similar.
Assume that \eqref{e:bcoeR} holds for all 
$\gamma \in \RPP$.
Combining \eqref{e:benv-bprox'} with \cref{p:12.9}\ref{p:12.9_left} yields
\begin{equation}
\inf\theta(X) \leq \theta\big(\bprox{\gamma\theta}(y)\big) \leq \benv[\gamma]{\theta}(y) \to \inf\theta(X) 
\quad\text{as}\quad \gamma \uparrow +\infty,
\end{equation} 
which implies that $\theta(\bprox{\gamma\theta}(y)) \to 
\inf\theta(X)$ as $\gamma \uparrow +\infty$. 
\end{proof}

\begin{theorem}
\label{p:limprox+}
Let $\theta\in\Gamma_0(X)$ be coercive such that 
$U\cap \dom\theta\neq\varnothing$ and let $x, y \in U$.
Then the following hold:
\begin{enumerate}
\item 
\label{p:limprox+_left}
The net $(\bprox{\gamma\theta}(y))_{\gamma \in \RPP}$ is bounded 
with all cluster points as $\gamma \uparrow +\infty$ lying in $\argmin\theta$. 
Moreover,
	\begin{enumerate}
	\item\label{p:limprox+_leftA} 
	if $\argmin\theta$ is a singleton, 
	then $\bprox{\gamma\theta}(y) \to \argmin\theta$ as $\gamma \uparrow +\infty$;
	\item\label{p:limprox+_leftP} 
	if $\argmin\theta \subseteq U$, 
	then $\bprox{\gamma\theta}(y) \to \bproj{\argmin\theta}y$ as $\gamma \uparrow +\infty$. 
	\end{enumerate}

\item 
\label{p:limprox+_right}
The net $(\fprox{\gamma\theta}(x))_{\gamma \in \RPP}$ is bounded 
with all cluster points as $\gamma \uparrow +\infty$ lying in $\argmin\theta$. 
Moreover,
	\begin{enumerate}
	\item 
	if $\argmin\theta$ is a singleton, 
	then $\fprox{\gamma\theta}(x) \to \argmin\theta$ as $\gamma \uparrow +\infty$;
	\item\label{p:limprox+_rightP} 
	if $\argmin\theta \subseteq U$, 
	then $\fprox{\gamma\theta}(x) \to \fproj{\argmin\theta}x$ as $\gamma \uparrow +\infty$.  
	\end{enumerate}
\end{enumerate}	
\end{theorem}
\begin{proof}
First, by assumption, \cite[Proposition~11.15(i)]{BC17} gives $\argmin\theta \neq\varnothing$. 
This combined with \cite[Lemma~1.24 and Corollary~8.5]{BC17} implies that 
$\argmin\theta =\lev{\inf\theta(X)}\theta$ is a nonempty closed convex subset of $X$.
Now since $\theta$ is coercive and since $D_f\geq 0$, we immediately get that  
\eqref{e:bcoeR} and \eqref{e:fcoeR} hold for all $\gamma \in \RPP$.

\ref{p:limprox+_left}: It follows from \eqref{e:benv-bprox'} that   
\begin{equation}
\label{e:bounded}
(\forall \gamma \in \RPP)\quad \bprox{\gamma\theta}(y) \in \lev{\theta(y)}\theta,
\end{equation}
and then
from the coercivity of $\theta$ and \cite[Proposition~11.12]{BC17} that $(\bprox{\gamma\theta}(y))_{\gamma \in \RPP}$ is bounded.
In turn, \cref{p:limthetaprox}\ref{p:limthetaprox_left} and \cref{l:cvg} imply that
all cluster points of $(\bprox{\gamma\theta}(y))_{\gamma \in \RPP}$ as $\gamma \uparrow +\infty$ lie in $\argmin\theta$, 
and we get \ref{p:limprox+_leftA}.

Now assume that $\argmin\theta \subseteq U$. 
Let $y'$ be a cluster point of $(\bprox{\gamma\theta}(y))_{\gamma \in \RPP}$ as $\gamma \uparrow +\infty$.  
Then $y' \in \argmin\theta \subseteq U$ and there exists a sequence $(\gamma_n)_\nnn$ in $\RPP$ such that 
$\gamma_n \uparrow +\infty$ and $\bprox{\gamma_n\theta}(y) \to y'$ as $n \to +\infty$.
Let $z \in \argmin\theta$.  
We have $(\forall\nnn)$ $\theta(z) \leq \theta(\bprox{\gamma_n\theta}(y))$, 
and by Proposition~\ref{p:prox}\ref{p:prox_left},
\begin{equation}
\scal{\nabla f(y) -\nabla f\big(\bprox{\gamma_n\theta}(y)\big)}{z -\bprox{\gamma_n\theta}(y)} 
\leq \gamma_n\Big(\theta(z)
-\theta\big(\bprox{\gamma_n\theta}(y)\big)\Big) \leq 0.
\end{equation}
Taking the limit as $n \to +\infty$ and using the continuity of $\nabla f$ yield
\begin{equation}
\scal{\nabla f(y) -\nabla f(y')}{z -y'} \leq 0.
\end{equation}
Since $z \in \argmin\theta$ was chosen arbitrarily 
and since $\argmin\theta$ is a closed convex subset of $X$ with 
$U\cap \argmin\theta \neq\varnothing$, 
in view of \cref{c:BregProj}\ref{c:BregProj_left}, $y' =\bproj{\argmin\theta}y$, 
and $\bproj{\argmin\theta}y$ is thus the only cluster point of $(\bprox{\gamma\theta}(y))_{\gamma \in \RPP}$ 
as $\gamma \uparrow +\infty$.
Hence, \ref{p:limprox+_leftP} holds.

\ref{p:limprox+_right}: 
The proof is similar to the one of \ref{p:limprox+_left}.
\end{proof}

\begin{remark}
Suppose that $f =\frac{1}{2}\|\cdot\|^2$ and let $\theta\in\Gamma_0(X)$ be coercive. 
By \cref{r:BregProj} and \cref{p:limprox+}, 
\begin{equation}
(\forall x \in X)\quad \prox_{\gamma\theta}(x) \to P_{\argmin\theta}x \quad\text{and}\quad \gamma \uparrow +\infty.
\end{equation}
\end{remark}

\begin{corollary}
Let $\theta\in\Gamma_0(\RR)$ be coercive such that $\argmin\theta
\subseteq U$ and let $z \in U$.
Then $\bprox{\gamma\theta}(z) \to P_{\argmin\theta}z$ and 
$\fprox{\gamma\theta}(z) \to P_{\argmin\theta}z$ 
as $\gamma \uparrow +\infty$.  
\end{corollary}
\begin{proof}
As shown in the proof of \cref{p:limprox+}, $\argmin\theta$ is a
nonempty closed convex subset of $X$ and hence 
$U\cap \argmin\theta\neq\varnothing$.
It now suffices to apply
\cref{p:limprox+}\ref{p:limprox+_leftP} and \ref{p:limprox+_rightP}
and to use \cref{p:allproj}.
\end{proof}

\section{Examples and function minimization}

\label{s:vier}

In this final section, we illustrate our theory by considering the
case in which $\theta$ is the \emph{nonsmooth} function
$x\mapsto |x-\tfrac{1}{2}|$. 

\begin{example}
\label{ex:illus}
Suppose that $X =\RR$ and let $\theta\colon \RR\to \RR\colon x\mapsto |x -\frac{1}{2}|$. 
Then $\theta \in \Gamma_0(X)$, $\dom\theta =X$, and $\theta$ is coercive with $\argmin\theta =\{\frac{1}{2}\}$. 
It follows that $U\cap \dom\theta  =U \neq\varnothing$ and, 
by \cref{f:coeR}, 
the assumptions \eqref{e:bcoeR} and \eqref{e:fcoeR} hold for all $\gamma \in \RPP$.
We revisit \cref{ex:examples} (with $J=1$) 
to illustrate \cref{p:min}, \cref{p:limprox}, and \cref{p:limprox+}. 
Let $\gamma \in \RPP$. 
We recall from \cref{p:prox} that 
\begin{equation}
\label{e:bprox}
\bprox{\gamma\theta}=(\nabla f +\gamma\partial\theta)^{-1}\circ \nabla f 
\end{equation}
and that
\begin{equation}
\label{e:fprox}
(\forall (x, y) \in U\times U)\quad 
y =\fprox{\gamma\theta}(x) \quad\Leftrightarrow\quad 0 \in \gamma\partial\theta(y) +\nabla^2 f(y)(y -x).
\end{equation}
Note that 
\begin{equation}
\partial\theta(x) =\begin{cases}
1, &\text{~if~} x >\tfrac{1}{2}; \\[+2mm]
-1, &\text{~if~} x <\tfrac{1}{2}; \\[+2mm]
\left[-1, 1\right], &\text{~if~} x =\tfrac{1}{2}. 
\end{cases}
\end{equation}

\begin{enumerate}
\item
\label{ex:illus_energy} 
\emph{Energy:}
Suppose that $f$ is the energy.
Then $U =\intdom f =\RR$. 
Since $\nabla f =\Id$, 
by \cref{r:energy} and \eqref{e:bprox}, 
$\bprox{\gamma\theta} =\fprox{\gamma\theta} =
(\Id +\gamma \partial\theta)^{-1}$.
We have that
\begin{equation}
\big(\Id +\gamma \partial\theta\big)(x) =\begin{cases}
x -\gamma, &\text{~if~} x <\tfrac{1}{2}; \\[+2mm]
x +\gamma, &\text{~if~} x >\tfrac{1}{2}; \\[+2mm]
\left[\tfrac{1}{2} -\gamma, \tfrac{1}{2}+\gamma\right], &\text{~if~} x =\tfrac{1}{2}.
\end{cases}
\end{equation}
Then $(\nabla f +\gamma \partial\theta)^{-1}(y)$ amounts to solving $(\nabla f +\gamma \partial\theta)(x) =y$ piecewise. 
For example, solving $x -\gamma =y$ for $x <\tfrac{1}{2}$ yields $x =y +\gamma$ for $y +\gamma <\tfrac{1}{2}$, 
so $(\nabla f +\gamma \theta)^{-1}(y) =y +\gamma$ for $y <\tfrac{1}{2} -\gamma$. 
Continuing in this fashion,
\begin{equation}
\bprox{\gamma\theta}(y) =
\fprox{\gamma\theta}(y) =(\Id +\gamma \partial\theta)^{-1}(y) 
=\begin{cases}
y +\gamma, &\text{~if~} y <\tfrac{1}{2} -\gamma; \\[+2mm]
y -\gamma, &\text{~if~} y >\tfrac{1}{2} +\gamma; \\[+2mm]
\tfrac{1}{2}, & \text{otherwise}
\end{cases}
\end{equation}
and by \eqref{e:benv-bprox}, 
\begin{equation}
\benv[\gamma]{\theta}(y) =\fenv[\gamma]{\theta}(y)
=\begin{cases}
-y +\tfrac{1-\gamma}{2}, &\text{~if~} y <\tfrac{1}{2} -\gamma;
\\[+2mm]
y -\tfrac{1+\gamma}{2}, &\text{~if~} y >\tfrac{1}{2} +\gamma;
\\[+2mm]
\tfrac{4y^2 -4y +1}{8\gamma}, &\text{~otherwise}.
\end{cases}
\end{equation}
It is clear that $\bprox{\gamma\theta}(\frac{1}{2}) =\frac{1}{2}$, 
while $(\forall y \in \RR\smallsetminus \{\frac{1}{2}\})$ $\bprox{\gamma\theta}(y) \neq y$, 
and so $\Fix\bprox{\gamma\theta} =\{\frac{1}{2}\} =\argmin\theta$.
As expected, $\bprox{\gamma\theta}(y) \to y$ as $\gamma \downarrow 0$, 
and $\bprox{\gamma\theta}(y) \to \frac{1}{2} =\argmin\theta$ as $\gamma \uparrow +\infty$. 
Moreover, $\benv[\gamma]{\theta}(y)\to\theta(y)$ as
$\theta\downarrow 0$; 
this is illustrated in \cref{fig:energy}.

\begin{figure}
	\begin{center}
		\vspace*{\fill}
		\adjustbox{trim={.15\width} {.1\height} {0.15\width} {.2\height},clip}%
		{\includegraphics[width=.41\columnwidth]{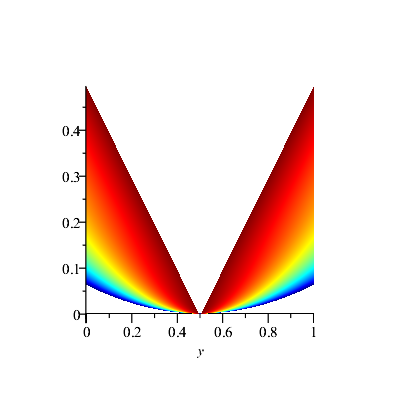}}
		\vspace*{\fill}
		\includegraphics[width=.1\columnwidth]{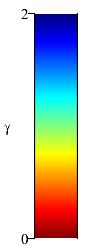}
		\vspace*{\fill}
		\includegraphics[width=.39\columnwidth]{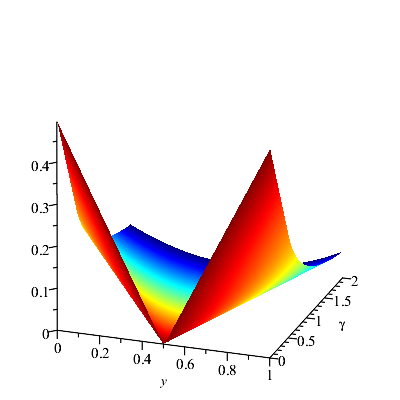}
		\vspace*{\fill}
	\end{center}
	
	\caption{Bregman envelope from \cref{ex:illus}\ref{ex:illus_energy}} \label{fig:energy}
\end{figure}

\item
\label{ex:illus_KL}
\emph{Boltzmann--Shannon entropy:} 
Suppose that $f$ is the Boltzmann--Shannon entropy.
Then $\dom f =\RP$, $U =\intdom f =\RPP$, 
$\nabla f(x) =\ln x$, and $\nabla^2f(x) =1/x$.
Again employing \eqref{e:bprox} 
and \eqref{e:benv-bprox}, we have
\begin{subequations}
\begin{align}
(\nabla f +\gamma \partial \theta)^{-1}(y) &=\begin{cases}
\exp(y +\gamma), &\text{~if~} y <-\ln 2 -\gamma; \\[+2mm]
\exp(y -\gamma), &\text{~if~} y >-\ln 2 +\gamma; \\[+2mm]
\tfrac{1}{2}, &\text{~otherwise}, 
\end{cases}\\[+4mm]
\bprox{\gamma\theta}(y) &=\begin{cases}
y\exp(\gamma), &\text{~if~} 0 <y <\tfrac{1}{2}\exp(-\gamma);
\\[+2mm]
y\exp(-\gamma), &\text{~if~} y >\tfrac{1}{2}\exp(\gamma);
\\[+2mm]
\tfrac{1}{2}, &\text{~otherwise},
\end{cases}\\[+4mm]
\benv[\gamma]{\theta}(y) &=\begin{cases}
\tfrac{y(1-e^\gamma)}{\gamma} + \tfrac{1}{2}, &\text{~if~} 0 <y
<\tfrac{1}{2}\exp(-\gamma); \\[+2mm]
\tfrac{y(1-e^{-\gamma})}{\gamma} - \tfrac{1}{2}, &\text{~if~}
\tfrac{1}{2}\exp(\gamma) <y; \\[+2mm]
\tfrac{2y- \ln(y) -1 -\ln(2)}{2\gamma}, &\text{~otherwise}.
\end{cases}
\end{align}
\end{subequations}
Clearly $\Fix\bprox{\gamma\theta} =\{\frac{1}{2}\} =\argmin\theta$.
It can also be seen that $\bprox{\gamma\theta}(y) \to y$ as $\gamma \downarrow 0$, 
and $\bprox{\gamma\theta}(y) \to \frac{1}{2} =\argmin\theta$ as $\gamma \uparrow +\infty$. 
Moreover, once again $\benv[\gamma]{\theta}(y)\to\theta(y)$ as
$\theta\downarrow 0$. 
This example is illustrated in \cref{fig:entropy}.

\begin{figure}
	\begin{center}
	\vspace*{\fill}
	\adjustbox{trim={.15\width} {.1\height} {0.15\width} {.15\height},clip}%
	{\includegraphics[width=.43\columnwidth]{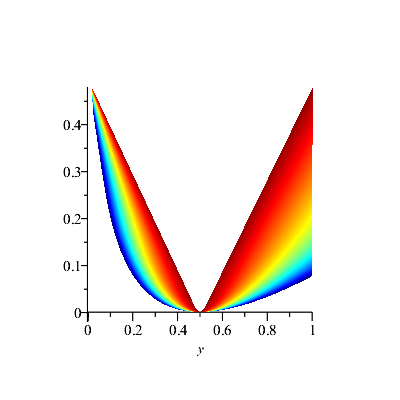}}
	\vspace*{\fill}
	\includegraphics[width=.1\columnwidth]{env_key}
	\vspace*{\fill}
	\includegraphics[width=.37\columnwidth]{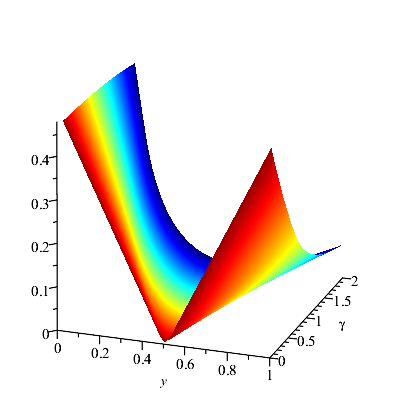}
	\vspace*{\fill}
	\end{center}
	\caption{Left Bregman envelope from \cref{ex:illus}\ref{ex:illus_KL}} \label{fig:entropy}
\end{figure}

Now \eqref{e:fprox} implies that for every $(x, y) \in \RPP\times \RPP$,
\begin{equation}
y =\fprox{\gamma\theta}(x) \quad\Leftrightarrow\quad 0 \in \gamma\partial\theta(y) +\frac{1}{y}(y -x)
\quad\Leftrightarrow\quad x \in y\big(1 +\gamma\partial\theta(y)\big).
\end{equation}
Solving the induced system of equations yields
\begin{equation}
\fprox{\gamma\theta}(x) =\begin{cases}
\tfrac{x}{1 -\gamma}, &\text{~if~} 0 <x <\tfrac{1 -\gamma}{2};
\\[+2mm]
\tfrac{x}{1 +\gamma}, &\text{~if~} x >\tfrac{1 +\gamma}{2};
\\[+2mm]
\tfrac{1}{2}, &\text{~otherwise}.
\end{cases}
\end{equation}
Using \eqref{e:fenv-fprox} and noting that $\frac{x}{1 -\gamma} <\frac{1}{2}$ if $x <\frac{1 -\gamma}{2}$ 
and $\frac{x}{1 +\gamma} >\frac{1}{2}$ if $x >\frac{1 +\gamma}{2}$, we obtain
\begin{equation}
\fenv[\gamma]{\theta}(x) =\begin{cases}
\tfrac{\ln(1 -\gamma)}{\gamma}x + \tfrac{1}{2}, &\text{~if~} 0 <x
<\tfrac{1 -\gamma}{2}; \\[+2mm]
\tfrac{\ln(1 +\gamma)}{\gamma}x - \tfrac{1}{2}, &\text{~if~} x
>\tfrac{1 +\gamma}{2}; \\[+2mm]
\tfrac{1}{\gamma}\left(x\ln(2x) -x +\tfrac{1}{2}\right), &\text{~otherwise}.
\end{cases}
\end{equation}
The right envelope is shown in \cref{fig:rightentropy}.

\begin{figure}
	\begin{center}
		\vspace*{\fill}
		\adjustbox{trim={.15\width} {.09\height} {0.15\width} {.15\height},clip}%
		{\includegraphics[width=.41\columnwidth]{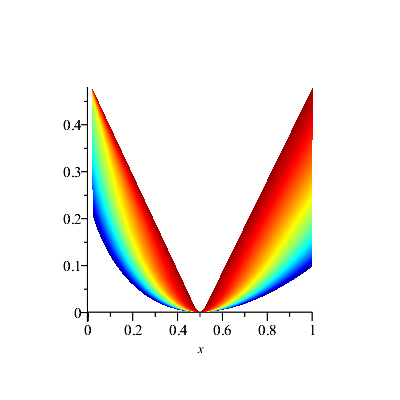}}
		\vspace*{\fill}
		\includegraphics[width=.1\columnwidth]{env_key}
		\vspace*{\fill}
		{\includegraphics[width=.39\columnwidth]{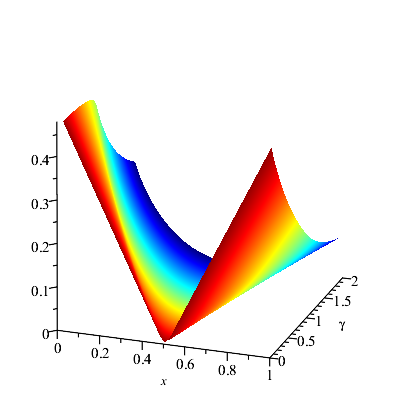}}
		\vspace*{\fill}
	\end{center}
	\caption{Right Bregman envelope from \cref{ex:illus}\ref{ex:illus_KL}}
	\label{fig:rightentropy}
\end{figure}

\item 
\label{ex:illus_FD}
\emph{Fermi--Dirac entropy:} 
Suppose that $f$ is the Fermi--Dirac entropy.
Then $\dom f =\left[0, 1\right]$, $U =\intdom f =\left]0, 1\right[$, 
$\nabla f(x) =\ln\left(\frac{x}{1 -x}\right)$, and $\nabla^2 f(x) =\frac{1}{x(1 -x)}$. 
Again by \eqref{e:bprox},  
\begin{subequations}
\begin{align}
(\nabla f +\gamma \partial\theta)^{-1}(y) &=\begin{cases}
\tfrac{\exp(y+\gamma)}{\exp(y+\gamma)+1}, &\text{~if~} y
<-\gamma; \\[+2mm]
\tfrac{\exp(y-\gamma)}{\exp(y-\gamma)+1}, &\text{~if~} y >\gamma;
\\[+2mm]
\tfrac{1}{2}, &\text{~otherwise},
\end{cases}\\[+4mm]
\bprox{\gamma\theta}(y) &=\begin{cases}
\tfrac{y\exp(\gamma)}{y\exp(\gamma) +1 -y} &\text{~if~} 0 < y
<\tfrac{\exp(-\gamma)}{1 +\exp(-\gamma)}; \\[+2mm]
\tfrac{y\exp(-\gamma)}{y\exp(-\gamma) +1 -y}, &\text{~if~}
\tfrac{\exp(\gamma)}{1 +\exp(\gamma)} <y < 1; \\[+2mm]
\tfrac{1}{2}, &\text{~otherwise}.
\end{cases}
\end{align}
\end{subequations}
A formula for $\benv[\gamma]{\theta}$ may be once again obtained
by using \eqref{e:benv-bprox}: 
\begin{align}
\benv[]{\gamma\theta}=\begin{cases}
-\frac{2\ln(y\exp(\gamma)-y+1)-\gamma}{2\gamma}, &\text{~if~}
0<y<\frac{\exp(-\gamma)}{1+\exp(-\gamma)};\\[+2mm]
-\frac{2\ln(y\exp(-\gamma)-y+1)+\gamma}{2\gamma}, 
&\text{~if~} 1 > y > \frac{\exp(\gamma)}{1+\exp(\gamma)};\\[+2mm]
-\frac{2\ln(2)+\ln(1-y)+\ln(y)}{2\gamma}, 
&\text{~if~} \frac{\exp(-\gamma)}{1+\exp(-\gamma)}\leq y \leq \frac{\exp(\gamma)}{1+\exp(\gamma)}.
\end{cases}
\end{align}

We illustrate this envelope in \cref{fig:fd}.

\begin{figure}
	\begin{center}
	\vspace*{\fill}
	\adjustbox{trim={.15\width} {.1\height} {0.15\width} {.15\height},clip}%
	{\includegraphics[width=.43\columnwidth]{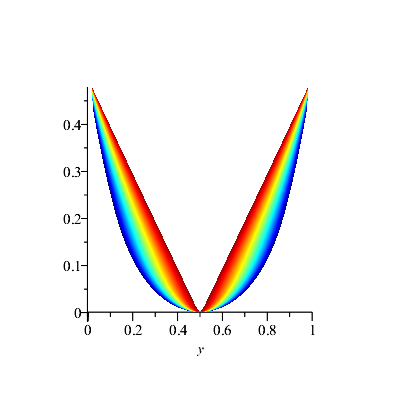}}
	\vspace*{\fill}
	\includegraphics[width=.1\columnwidth]{env_key}
	\vspace*{\fill}
	\includegraphics[width=.37\columnwidth]{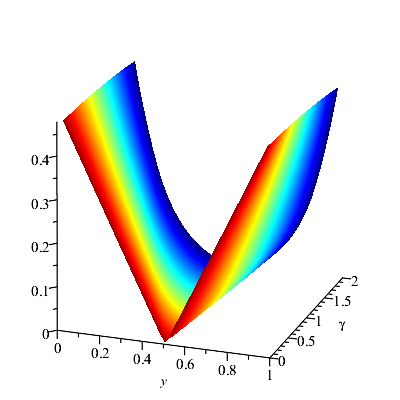}
	\vspace*{\fill}
	\end{center}
	\caption{Left Bregman envelope from \cref{ex:illus}\ref{ex:illus_FD}} \label{fig:fd}
\end{figure}

Next we have from \eqref{e:fprox} that for every $(x, y) \in \left]0, 1\right[\times \left]0, 1\right[$,
\begin{equation}
y =\fprox{\gamma\theta}(x) \quad\Leftrightarrow\quad 0 \in \gamma\partial\theta(y) +\frac{1}{y(1 -y)}(y -x)
\quad\Leftrightarrow\quad x \in \gamma y(1 -y)\partial\theta(y) +y.
\end{equation}
Solving the induced system of equations gives
\begin{equation}
\fprox{\gamma\theta}(x) =\begin{cases}
\tfrac{\gamma -1 +\sqrt{(\gamma -1)^2 +4\gamma x}}{2\gamma},
&\text{~if~} 0 <x <\tfrac{2 -\gamma}{4}; \\[+2mm]
\tfrac{\gamma +1 -\sqrt{(\gamma +1)^2 -4\gamma x}}{2\gamma},
&\text{~if~} \tfrac{2 +\gamma}{4} <x <1; \\[+2mm]
\tfrac{1}{2}, &\text{~otherwise}
\end{cases}
\end{equation}
and, in turn, \eqref{e:fenv-fprox} gives
\begin{equation}
\fenv{\gamma\theta}(x)=\begin{cases}
\frac{2\ln\left( \frac{2\gamma x^x
(\gamma + 1-\sqrt{\gamma^2+4\gamma x -2\gamma +1})^{x-1}
}{(\gamma-1+\sqrt{\gamma^2+4\gamma x -2\gamma +1})^x
(1-x)^{x-1}}\right) +1-\sqrt{\gamma^2+4\gamma x -2\gamma
+1}}{2\gamma}, &\text{~if~} 0<x<\frac{2-\gamma}{4};\\[+2mm]
\frac{2\ln\left(\frac{2\gamma x^x (\gamma
-1+\sqrt{\gamma^2-4\gamma x +2\gamma
+1})^{x-1}}{(1-x)^{x-1}(\gamma +1-\sqrt{\gamma^2-4\gamma x
+2\gamma +1})^x} \right) +1-\sqrt{\gamma^2-4\gamma x +2\gamma
+1}}{2\gamma}, &\text{~if~} \tfrac{2 +\gamma}{4} <x <1; \\[+2mm]
\frac{x\ln(x)+(1-x)\ln(1-x)+\ln(2)}{\gamma}, &\text{~otherwise}. 
\end{cases}
\end{equation}
The right envelope is shown in \cref{fig:rightfd}.
\begin{figure}
	\begin{center}
		\vspace*{\fill}
		\adjustbox{trim={.15\width} {.09\height} {0.15\width} {.15\height},clip}%
		{\includegraphics[width=.43\columnwidth]{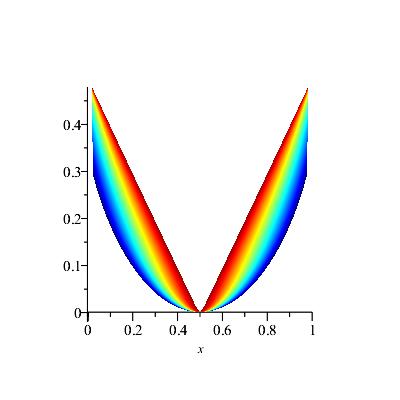}}
		\vspace*{\fill}
		\includegraphics[width=.1\columnwidth]{env_key}
		\vspace*{\fill}
		\includegraphics[width=.37\columnwidth]{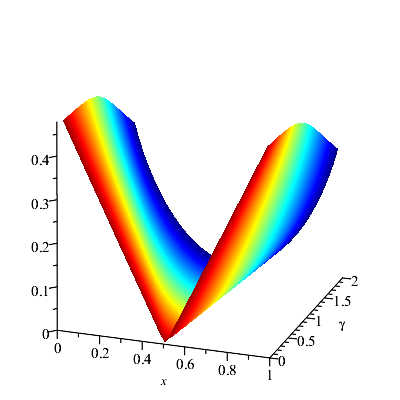}
		\vspace*{\fill}
	\end{center}
	\caption{Right Bregman envelope from \cref{ex:illus}\ref{ex:illus_FD}}
	\label{fig:rightfd}
\end{figure}

\end{enumerate}
\end{example}

We conclude this section with some remarks concerning the
minimization of the (nonlinear) functional 
\begin{align}
	I_\tau \colon L^1[0,1] \to\RR\colon 
	 x \mapsto \int_{0}^{1} \tau\big(x(s)\big) {\rm d}s, 
\end{align} 
where $\tau$ is a convex, lower semicontinuous, and proper, and
subject to finitely many constraints 
\begin{equation}
	\langle a_k,x \rangle = \int_0^1 a_k(s)x(s) {\rm d}s =
	b_k, \quad \text{for $k\in\{1,\ldots,n\}$}, 
\end{equation}
and where $a_k \in L^\infty[0,1]$ and $\rho$ was used to generate the 
(consistent) data: $\scal{a_k}{\rho}=b_k$. 
Under appropriate assumptions (for details, see \cite{BL2000}, \cite{BL1991}, \cite[Section 7]{BL2016}, \cite[Theorem 6.3.4]{BV2010}, \cite{BY2014}, \cite[Section 4.7]{BZ2005}), 
recovering the (primal) solution $x = x_\tau$ amounts to first obtaining a
dual solution by solving the finite system of nonlinear equations
\begin{equation}\label{optimalmultipliers}
	\int_0^1 (\tau^*)' \left(\sum_{j=1}^n \mu_j a_j(s)
	\right) a_k(s) {\rm d}s = b_k,\quad \text{where~}
	k\in\{1,\ldots,n\},
\end{equation}
followed by computing
\begin{equation}\label{eqn:primal}
x_\tau = \big(\tau^*\big)' \Big(\sum_{j=1}^n \mu_j a_j(s) \Big).
\end{equation}
\begin{figure}
	\begin{center}
		\vspace*{\fill}
		\adjustbox{trim={.15\width} {.09\height} {0.15\width} {.15\height},clip}%
		{\includegraphics[width=.5\columnwidth]{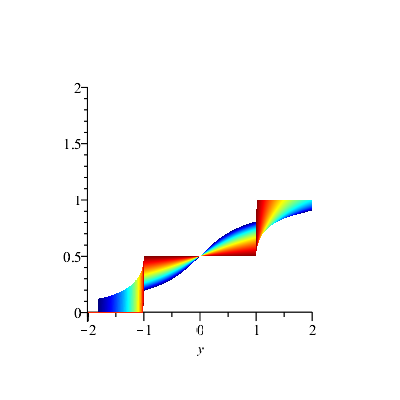}}
		\vspace*{\fill}
		\includegraphics[width=.1\columnwidth]{env_key}
		\vspace*{\fill}
		\includegraphics[width=.42\columnwidth]{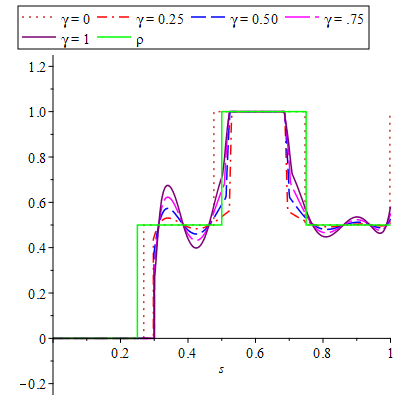}
		\vspace*{\fill}
	\end{center}
	\caption{$(\benv[\gamma]{\theta})^{*\prime}$ for $\gamma \in [0,2]$ (left). Primal solutions \eqref{eqn:primal} for selected $\gamma$ values in $[0,1]$ (right).}\label{fig:entropyminimization}
\end{figure}
As a numerical illustration, we assume that 
$\tau = \benv[\gamma]{\theta}$ is the left Bregman envelope from 
 \cref{ex:illus}\ref{ex:illus_FD} and $\rho$ is the step function
 pictured in Figure~\ref{fig:entropyminimization} (right). We clearly observe the influence of the parameter $\gamma$; smaller values of $\gamma$ lead to primal solutions that are closer to the step function that was used to generate the data. Were a different, smoother $\rho$ to be used, a larger value of $\gamma$ might be more appropriate. The reason for the varying extent to which the primal solutions for different choices of $\gamma$ resemble step functions may be gleaned from Figure~\ref{fig:entropyminimization} (left), where $(\tau^*)'$---upon which the primal solution \eqref{eqn:primal} depends---is shown.

Similarly attempting to compute with $\tau$ as the right envelope from \cref{ex:illus}\ref{ex:illus_FD}, we are unable to symbolically invert the gradient. We leave the numerical attempt at inversion as future work.

\subsection*{Acknowledgments}
HHB was partially supported by the Natural Sciences and
Engineering Research Council of Canada.
MND was partially supported by the Australian Research Council Discovery Project DP160101537.


\begin{thebibliography}{99}

\sepp

\bibitem{ACM}
F.\ Alvarez, R.\ Correa, and M.\ Marechal,
Regular self-proximal distances are Bregman,
\emph{Journal of Convex Analysis}~24 (2017), 135--148. 

\bibitem{Attouch77}
H.\ Attouch,
Convergence de fonctions convexes, des sous-diff\'erentiels et semi-groupes associ\'es,
\emph{Comptes Rendus de l'Acad\'emie des Sciences de Paris}~284 (1977), 539--542. 

\bibitem{Attouch84}
H.\ Attouch,
\emph{Variational Convergence for Functions and Operators},
Pitman, 1984. 

\bibitem{BBT}
H.H.\ Bauschke, J.\ Bolte, and M.\ Teboulle,
A descent lemma beyond Lipschitz gradient continuity:
first-order methods revisited and applications,
\emph{Mathematics of Operations Research}~42 (2016), 330--348.

\bibitem{BB97}
H.H.\ Bauschke and J.M.\ Borwein,
Legendre functions and the method of random Bregman projections,
\emph{Journal of Convex Analysis}~4 (1997), 27--67. 

\bibitem{BBC03}
H.H.\ Bauschke, J.M.\ Borwein, and P.L.\ Combettes,
Bregman monotone optimization algorithms,
\emph{SIAM Journal on Control and Optimization}~42 (2003), 596--636. 

\bibitem{BC03}
H.H.\ Bauschke and P.L.\ Combettes,
Iterating Bregman retractions,
\emph{SIAM Journal on Optimization}~13 (2003), 1159--1173.

\bibitem{BC17}
H.H.\ Bauschke and P.L.\ Combettes,
\emph{Convex Analysis and Monotone Operator Theory in Hilbert Spaces},
second edition, 
Springer, 2017.

\bibitem{BCN06} 
H.H.\ Bauschke, P.L.\ Combettes, and D.\ Noll, 
Joint minimization with alternating Bregman proximity operators,
\emph{Pacific Journal of Optimization}~2 (2006), 401--424.

\bibitem{BauLew}
H.H.\ Bauschke and A.S.\ Lewis,
Dykstra's algorithm with Bregman projectors: a convergence proof,
\emph{Optimization}~48 (2000), 409--427.

\bibitem{BN02}
H.H.\ Bauschke and D.\ Noll,
The method of forward projections,
\emph{Journal of Nonlinear and Convex Analysis}~3 (2002), 191--205.

\bibitem{BL2000} 
J.M. Borwein and A.S. Lewis,  
\emph{Convex Analysis and Nonlinear Optimization: Theory and Examples}, 
second edition
Springer, 2006.

\bibitem{BL1991} 
J.M. Borwein and A.S. Lewis, 
Duality relationships for entropy--like minimization problems, 
\emph{SIAM Journal on Control and Optimization}~29 (1991), 325--338.

\bibitem{BL2016} 
J.M. Borwein and S.B. Lindstrom, 
Meetings with Lambert $\mathcal W$ and other special functions in optimization and analysis, 
\emph{Pure and Applied Functional Analysis}~1(3) (2017), 361--396. 

\bibitem{BV2010} 
J.M. Borwein and J. D. Vanderwerff, 
\emph{Convex Functions: Constructions, Characterizations and Counterexamples}, 
Cambridge University Press, 2010.

\bibitem{BY2014} 
J.M. Borwein and L. Yao, 
Legendre-type integrands and convex integral functions, 
\emph{Journal of Convex Analysis}~21 (2014), 264--288.

\bibitem{BZ2005} 
J.M. Borwein and Q. Zhu, 
\emph {Techniques of Variational Analysis}, 
Springer, 2005.

\bibitem{Bre67}
L.M.\ Bregman,
The relaxation method of finding the common point of convex sets and
its application to the solution of problems in convex programming,
\emph{USSR Computational Mathematics and Mathematical Physics}~7 (1967), 200--217.

\bibitem{BurKas12}
R.\ Burachik and G.\ Kassay,
On a generalized proximal point method for solving equilibrium problems in Banach spaces,
\emph{Nonlinear Analysis}~75 (2012), 6456--6464.

\bibitem{ButIus}
D.\ Butnariu and A.N.\ Iusem, 
\emph{Totally Convex Functions for Fixed Points Computation and Infinite Dimensional Optimization},
Kluwer, 2000.

\bibitem{BC01}
C.\ Byrne and Y.\ Censor,
Proximity function minimization using multiple Bregman projections, 
with applications to split feasibility and Kullback--Leibler distance minimization,
\emph{Annals of Operations Research}~105 (2001), 77--98. 

\bibitem{CH02}
Y.\ Censor and G.\ T.\ Herman,
Block-iterative algorithms with underrelaxed Bregman projections,
\emph{SIAM Journal on Optimization}~13 (2002), 283--297.

\bibitem{CR98}
Y.\ Censor and S.\ Reich, 
The Dykstra algorithm with Bregman projections, 
Communications in Applied Analysis~2 (1998), 407--419.

\bibitem{CZ92}
Y.\ Censor and S.A.\ Zenios,
Proximal minimization algorithm with D-functions,
\emph{Journal of Optimization Theory and Applications}~73 (1992), 451--464.

\bibitem{CZ97}
Y.\ Censor and S.A.\ Zenios,
\emph{Parallel Optimization: Theory, Algorithms, and Applications},
Oxford University Press, 1997.

\bibitem{CT93}
G.\ Chen and M.\ Teboulle,
Convergence analysis of a proximal-like minimization algorithm using Bregman functions,
\emph{SIAM Journal on Optimization}~3 (1993), 538--543.

\bibitem{CKS}
Y.Y.\ Chen, C.\ Kan, and W.\ Song,
The Moreau envelop function and proximal
mapping with respect to the Bregman distance in Banach spaces,
\emph{Vietnam Journal of Mathematics}~40 (2012), 181--199.

\bibitem{CN16}
P.L.\ Combettes and Q.V.\ Nguyen,
Solving composite monotone inclusions in reflexive Banach spaces 
by constructing best Bregman approximations from their Kuhn--Tucker set,
\emph{Journal of Convex Analysis}~23 (2016), 481--510. 

\bibitem{CR13}
P.L.\ Combettes and N.N.\ Reyes, 
Moreau's decomposition in Banach spaces, 
\emph{Mathematical Programming}, 139 (2013), 103--114.

\bibitem{Eck93} 
J.\ Eckstein,  
Nonlinear proximal point algorithms using Bregman functions, 
with applications to convex programming,
\emph{Mathematics of Operations Research}~18 (1993), 202--226.

\bibitem{HUL93}
J.-B.\ Hiriart-Urruty and C. Lemar\'echal, 
\emph{Convex Analysis and Minimization Algorithms II: Advanced Theory and Bundle Methods},
Springer, 1993.

\bibitem{KS}
C.\ Kan and W.\ Song,
The Moreau envelope function and proximal mapping
in the sense of the Bregman distance,
\emph{Nonlinear Analysis}~75 (2012), 1385--1399.

\bibitem{KRS}
G.\ Kassay, S.\ Reich, and S.\ Sabach,
Iterative methods for solving systems of variational inequalities
in reflexive Banach spaces, 
\emph{SIAM Journal on Optimization}~21 (2011), 1319--1344.

\bibitem{Kiw97}
K.C.\ Kiwiel,
Proximal minimization methods with generalized Bregman functions,
\emph{SIAM Journal on Control and Optimization}~35 (1997), 1142--1168.

\bibitem{Lan16}
K.\ Lange, 
\emph{MM Optimization Algorithms}, 
SIAM, 2016.

\bibitem{MorNam}
B.S.\ Mordukhovich and N.M.\ Nam,
\emph{An Easy Path to Convex Analysis and Applications},
Morgan \& Claypool Publishers, 2013. 

\bibitem{Mor62}
J.-J.\ Moreau, 
Fonctions convexes duales et points proximaux dans un espace hilbertien,
\emph{Comptes Rendus de l'Acad\'emie des Sciences}~255 (1962), 2897--2899.

\bibitem{Mor65}
J.-J. Moreau, Proximit\'{e} et dualit\'{e} dans un espace hilbertien, 
\emph{Bulletin de la Soci\'{e}t\'{e} Math\'{e}matique de France}~93 (1965), 273--299. 

\bibitem{Nguyen17}
Q.V.\ Nguyen,
Forward-backward splitting with Bregman distances,
\emph{Vietnam Journal of Mathematics}~45 (2017), 519--539.

\bibitem{Nguyen16}
Q.V.\ Nguyen,
Variable quasi-Bregman monotone sequences,
\emph{Numerical Algorithms}~73 (2016), 1107--1130.

\bibitem{Roc70}
R.T.\ Rockafellar,
\emph{Convex Analysis},
Princeton University Press, 1970.

\bibitem{RW98} 
R.T.\ Rockafellar and R.J-B\ Wets, 
\emph{Variational Analysis},
Springer-Verlag, 1998.

\bibitem{Sabach11}
S.\ Sabach, 
Products of finitely many resolvents of maximal monotone mappings
in reflexive Banach spaces,
\emph{SIAM Journal on Optimization}~21 (2011), 1289--1308. 

\end{thebibliography}
\end{document}